\documentclass[12pt]{amsart}       
\usepackage{txfonts}
\usepackage{amssymb}
\usepackage{eucal}
\usepackage{graphicx}
\usepackage{amsmath}
\usepackage{amscd}
\usepackage[all]{xy}           
\usepackage{amsfonts,latexsym}
\usepackage{xspace}
\usepackage{epsfig}

\usepackage{epstopdf}
\usepackage{subfig}
\usepackage{enumitem}
\usepackage{stmaryrd}
\usepackage{amsmath,amssymb,amsthm}

\usepackage{tikz-cd}
\usepackage{float}
\usepackage{color}
\usepackage{fancybox}
\usepackage{colordvi}
\usepackage{multicol}
\usepackage{colordvi}
\usepackage{tikz}
\usetikzlibrary{decorations.pathreplacing}
\usepackage{wasysym}
\usepackage[active]{srcltx} 
\ifpdf
  \usepackage[colorlinks,final,backref=page,hyperindex]{hyperref}
\else
  \usepackage[colorlinks,final,backref=page,hyperindex,hypertex]{hyperref}
\fi

\usepackage{extarrows}
\usepackage{enumitem}

\newcommand{\nc}{\newcommand}
\newcommand{\delete}[1]{}

\nc{\mlabel}[1]{\label{#1}}  
\nc{\mcite}[1]{\cite{#1}}  
\nc{\mref}[1]{\ref{#1}}  
\nc{\meqref}[1]{\eqref{#1}}  
\nc{\mbibitem}[1]{\bibitem{#1}} 

\delete{
\nc{\mlabel}[1]{\label{#1}  
{\hfill \hspace{1cm}{\small\tt{{\ }\hfill(#1)}}}}
\nc{\mcite}[1]{\cite{#1}{\small{\tt{{\ }(#1)}}}}  
\nc{\mref}[1]{\ref{#1}{{\tt{{\ }(#1)}}}}  
\nc{\meqref}[1]{\eqref{#1}{{\tt{{\ }(#1)}}}}  
\nc{\mbibitem}[1]{\bibitem[\bf #1]{#1}} 
}

\makeatletter

\newcommand{\Rmnum}[1]{\expandafter\@slowromancap\romannumeral #1@}
\makeatother

\newtheorem{theorem}{Theorem}[section]
\newtheorem{prop}[theorem]{Proposition}
\newtheorem{lemma}[theorem]{Lemma}
\newtheorem{coro}[theorem]{Corollary}
\theoremstyle{definition}
\newtheorem{defn}[theorem]{Definition}
\newtheorem{prop-def}{Proposition-Definition}[section]

\newtheorem{remark}[theorem]{Remark}

\newtheorem{conjecture}[theorem]{Conjecture}

\newtheorem{tempex}[theorem]{Example}
\newtheorem{tempexs}[theorem]{Examples}
\newenvironment{exam}{\begin{tempex}\rm}{\end{tempex}}

\nc{\Irr}{\mathrm{Irr}}
\nc{\ncrbw}{\calr}  
\nc{\NS}{U_{NS}}
\nc{\FN}{F_{\mathrm Nij}}
\nc{\dfgen}{V} \nc{\dfrel}{R}
\nc{\dfgenb}{\vec{v}} \nc{\dfrelb}{\vec{r}}
\nc{\dfgene}{v} \nc{\dfrele}{r}
\nc{\dfop}{\odot}
\nc{\dfoa}{\dfop^{(1)}} \nc{\dfob}{\dfop^{(2)}}
\nc{\dfoc}{\dfop^{(3)}} \nc{\dfod}{\dfop^{(4)}}
\nc{\mapm}[1]{\lfloor\!|{#1}|\!\rfloor}
\nc{\cmapm}[1]{\frakC(#1)}
\nc{\red}{\mathrm{Red}}
\nc{\cm}{C}
\nc{\supp}{\mathrm{Supp}}
\nc{\lex}{\mathrm{lex}}

\nc{\disp}[1]{\displaystyle{#1}}
\nc{\bin}[2]{ (_{\stackrel{\scs{#1}}{\scs{#2}}})}  
\nc{\bs}{\bar{S}} \nc{\ep}{\epsilon}
\nc{\dbigcup}{\stackrel{\bullet}{\bigcup}}
\nc{\la}{\longrightarrow} \nc{\cprod}{\ast} \nc{\rar}{\rightarrow}
\nc{\dar}{\downarrow} \nc{\labeq}[1]{\stackrel{#1}{=}}
\nc{\dap}[1]{\downarrow \rlap{$\scriptstyle{#1}$}}
\nc{\uap}[1]{\uparrow \rlap{$\scriptstyle{#1}$}}
\nc{\defeq}{\stackrel{\rm def}{=}} \nc{\dis}[1]{\displaystyle{#1}}
\nc{\dotcup}{\ \displaystyle{\bigcup^\bullet}\ }
\nc{\sdotcup}{\tiny{ \displaystyle{\bigcup^\bullet}\ }}
\nc{\fe}{\'{e}}
\nc{\hcm}{\ \hat{,}\ } \nc{\hcirc}{\hat{\circ}}
\nc{\hts}{\hat{\shpr}} \nc{\lts}{\stackrel{\leftarrow}{\shpr}}
\nc{\denshpr}{\den{\shpr}}
\nc{\rts}{\stackrel{\rightarrow}{\shpr}} \nc{\lleft}{[}
\nc{\lright}{]} \nc{\uni}[1]{\tilde{#1}} \nc{\free}[1]{\bar{#1}}
\nc{\freea}[1]{\tilde{#1}} \nc{\freev}[1]{\hat{#1}}
\nc{\dt}[1]{\hat{#1}}
\nc{\wor}[1]{\check{#1}}
\nc{\intg}[1]{F_C(#1)}
\nc{\den}[1]{\check{#1}} \nc{\lrpa}{\wr} \nc{\mprod}{\pm}
\nc{\dprod}{\ast_P} \nc{\curlyl}{\left \{ \begin{array}{c} {} \\
{} \end{array}
    \right .  \!\!\!\!\!\!\!}
\nc{\curlyr}{ \!\!\!\!\!\!\!
    \left . \begin{array}{c} {} \\ {} \end{array}
    \right \} }
\nc{\longmid}{\left | \begin{array}{c} {} \\ {} \end{array}
    \right . \!\!\!\!\!\!\!}
\nc{\lin}{\call} \nc{\ot}{\otimes}
\nc{\ora}[1]{\stackrel{#1}{\rar}}
\nc{\ola}[1]{\stackrel{#1}{\la}}
\nc{\scs}[1]{\scriptstyle{#1}} \nc{\mrm}[1]{{\rm #1}}
\nc{\margin}[1]{\marginpar{\rm #1}}   
\nc{\dirlim}{\displaystyle{\lim_{\longrightarrow}}\,}
\nc{\invlim}{\displaystyle{\lim_{\longleftarrow}}\,}
\nc{\mvp}{\vspace{0.5cm}}
\nc{\mult}{m}       
\nc{\svp}{\vspace{2cm}} \nc{\vp}{\vspace{8cm}}
\nc{\proofbegin}{\noindent{\bf Proof: }}
\nc{\proofend}{$\blacksquare$ \vspace{0.5cm}}
\nc{\sha}{{\mbox{\cyr X}}}  
\nc{\ncsha}{{\mbox{\cyr X}^{\mathrm NC}}}
\newfont{\scyr}{wncyr10 scaled 550}
\nc{\ssha}{\mbox{\bf \scyr X}}
\nc{\ncshao}{{\mbox{\cyr X}^{\mathrm NC,\,0}}}
\nc{\shpr}{\diamond}    
\nc{\shprc}{\shpr_c}
\nc{\shpro}{\diamond^0}    
\nc{\shpru}{\check{\diamond}} \nc{\spr}{\cdot}
\nc{\catpr}{\diamond_l} \nc{\rcatpr}{\diamond_r}
\nc{\lapr}{\diamond_a} \nc{\lepr}{\diamond_e} \nc{\sprod}{\bullet}
\nc{\un}{u}                 
\nc{\vep}{\varepsilon} \nc{\labs}{\mid\!} \nc{\rabs}{\!\mid}
\nc{\hsha}{\widehat{\sha}} \nc{\psha}{\sha^{+}} \nc{\tsha}{\tilde{\sha}}
\nc{\lsha}{\stackrel{\leftarrow}{\sha}}
\nc{\rsha}{\stackrel{\rightarrow}{\sha}} \nc{\lc}{\lfloor}
\nc{\rc}{\rfloor} \nc{\sqmon}[1]{\langle #1\rangle}
\nc{\altx}{\Lambda} \nc{\vecT}{\vec{T}} \nc{\piword}{{\mathfrak P}}
\nc{\lbar}[1]{\overline{#1}}
\nc{\dep}{\mathrm{dep}}

\nc{\id}{\mrm{id}}
\nc{\ab}{\mathbf{Ab}} \nc{\Alg}{\mathbf{Alg}}
\nc{\Bax}{\mathbf{CRB}} \nc{\Algo}{\mathbf{Alg}^0}
\nc{\cRB}{\mathbf{CRB}} \nc{\cRBo}{\mathbf{CRB}^0}
\nc{\RBo}{\mathbf{RB}^0} \nc{\BRB}{\mathbf{RB}}
\nc{\Dend}{\mathbf{DD}} \nc{\bfk}{{\bf k}} \nc{\bfone}{{\bf 1}}
\nc{\base}[1]{{a_{#1}}} \nc{\Cat}{\mathbf{Cat}}
 \nc{\DN}{\mathbf{DN}}
\nc{\NA}{\mathbf{NA}}
\nc{\SDN}{\mathbf{SDN}}
\nc{\Diff}{\mathbf{Diff}} \nc{\gap}{\marginpar{\bf
Incomplete}\noindent{\bf Incomplete!!}
    \svp}
\nc{\FMod}{\mathbf{FMod}} \nc{\Int}{\mathbf{Int}}
\nc{\Mon}{\mathbf{Mon}}
\nc{\RB}{\mathbf{RB}} \nc{\remarks}{\noindent{\bf Remarks: }}
\nc{\Rep}{\mathbf{Rep}} \nc{\Rings}{\mathbf{Rings}}
\nc{\Sets}{\mathbf{Sets}} \nc{\bfx}{\mathbf{x}}
\nc{\BA}{{\Bbb A}} \nc{\CC}{{\Bbb C}} \nc{\DD}{{\Bbb D}}
\nc{\EE}{{\Bbb E}} \nc{\FF}{{\Bbb F}} \nc{\GG}{{\Bbb G}}
\nc{\HH}{{\Bbb H}} \nc{\LL}{{\Bbb L}} \nc{\NN}{{\Bbb N}}
\nc{\QQ}{{\Bbb Q}} \nc{\RR}{{\Bbb R}} \nc{\TT}{{\Bbb T}}
\nc{\VV}{{\Bbb V}} \nc{\ZZ}{{\Bbb Z}}

\nc{\cala}{{\mathcal A}} \nc{\calb}{{\mathcal B}}
\nc{\calc}{{\mathcal C}}
\nc{\cald}{{\mathcal D}} \nc{\cale}{{\mathcal E}}
\nc{\calf}{{\mathcal F}} \nc{\calg}{{\mathcal G}}
\nc{\calh}{{\mathcal H}} \nc{\cali}{{\mathcal I}}
\nc{\calj}{{\mathcal J}} \nc{\call}{{\mathcal L}}
\nc{\calm}{{\mathcal M}} \nc{\caln}{{\mathcal N}}
\nc{\calo}{{\mathcal O}} \nc{\calp}{{\mathcal P}}
\nc{\calr}{{\mathcal R}} \nc{\cals}{{\mathcal S}} \nc{\calt}{{\mathcal T}}
\nc{\calw}{{\mathcal W}} \nc{\calx}{{\mathcal X}} \nc{\caly}{{\mathcal Y}} \nc{\calz}{{\mathcal Z}}
\nc{\CA}{\mathcal{A}}

\nc{\tred}[1]{\textcolor{red}{#1}} \nc{\tgreen}[1]{\textcolor{green}{#1}}
\nc{\tblue}[1]{\textcolor{blue}{#1}} \nc{\tpurple}[1]{\textcolor{purple}{#1}}

\nc{\li}[1]{\tpurple{\underline{Li:}#1 }}
\nc{\liadd}[1]{\tpurple{#1}}
\nc{\xing}[1]{\tblue{\underline{Xing:}#1 }}
\nc{\YZ}[1]{\tred{\underline{Yaozhou:} #1}}

\nc{\deleted}[1]{\delete{#1}}
\nc{\astarrow}{\overset{\raisebox{-2pt}{{\scriptsize $\ast$}}}{\rightarrow}}\nc{\tvarrow}[3]{#1\overset{(t,v)}{\longrightarrow}_{#3} #2}

\nc{\sym}{{\rm Sp_{2n}}(\mathbb{C})}
\nc{\dg}[1]{{\rm dg}(#1)}
\nc{\hl}[1]{{\rm hl}(#1)}
\nc{\co}[1]{\mathcal{C}_{#1}}
\nc{\Cco}{\mathcal{C}_{2n}^{c}}
\nc{\Cpar}{{\rm P}_c^{2n}}
\nc{\SMC}[3]{#1 \overset{#2}{\Rightarrow} #3}
\nc{\CSMT}[1]{\text{\rm CSMT}^n(#1)}
\nc{\trans}[2]{t_{#1,#2}^c}
\nc{\reW}[5]{W_{#1,#2}^{#3,#4,#5}}
\nc{\reM}[5]{M^{#1,#2}_{#3,#4,#5}}
\nc{\cin}[2]{\mathcal{I}_{#1,#2}^\circ}
\nc{\cout}[2]{\mathcal{O}_{#1,#2}^\circ}
\nc{\Cin}[2]{\mathcal{I}_{#1,#2}}
\nc{\Cout}[2]{\mathcal{O}_{#1,#2}}
\nc{\Coe}[2]{{\rm Coe}(#1)[#2]}
\nc{\RS}{{\rm rowsums}(m)}
\nc{\CS}{{\rm colsums}(m)}

\providecommand{\tilc}{\widetilde{C}_n}

\begin{document}
\title[Type C transpositions and strong marked tableaux]{Type C transpositions and strong marked tableaux}
%
\author{Yaozhou Fang}
\address{School of Mathematics and Statistics, Lanzhou University
Lanzhou, 730000, China
}
\email{fangyzh2024@lzu.edu.cn}

\author{Xing Gao$^*$}\thanks{*Corresponding author}
\address{School of Mathematics and Statistics, Lanzhou University,
Lanzhou, 730000, China; Gansu Provincial Research Center for Basic Disciplines of Mathematics and Statistics, Lanzhou, 730070, China
}
\email{gaoxing@lzu.edu.cn}

\date{\today}
\begin{abstract}
Motivated by the Morse--Seelinger conjecture on type C $k$-Schur functions,
we develop a permutation model for type C strong marked tableaux.
Using type C transpositions, we give an explicit correspondence with
the symmetric-core model that preserves the marking data. We also
describe the induced affine Weyl group action on bounded partitions.
\end{abstract}

\makeatletter
\@namedef{subjclassname@2020}{\textup{2020} Mathematics Subject Classification}
\makeatother
\subjclass[2020]{
05E05, 
05E10, 
14N15, 
}

\keywords{affine Weyl group of Type C, Cores, Type C $k$-Schur functions}

\maketitle

\tableofcontents

\setcounter{section}{0}

\allowdisplaybreaks

\section{Introduction}

\subsection{Affine Schubert calculus and the tableau problem}

A central problem in affine Schubert calculus is to represent Schubert
classes by symmetric functions whose coefficients have explicit
combinatorial interpretations. For the affine Grassmannian of
$SL_{k+1}(\mathbb{C})$, Lam identified the homology Schubert basis with
$k$-Schur functions~\cite{Lam}. The affine insertion theory of Lam,
Lapointe, Morse, and Shimozono connected these functions with strong
marked tableaux and Pieri rules~\cite{LLMS}. In this setting, affine
permutations, cores, and bounded partitions provide complementary
combinatorial models: permutations record the Coxeter-theoretic
structure, while cores and partitions describe the shapes and growth
underlying the tableaux.

For the symplectic group $\sym$, Lam, Schilling, and Shimozono constructed
Hopf-algebra isomorphisms
\[
H_*(Gr_{\sym})\cong\Gamma_{(n)},
\qquad
H^*(Gr_{\sym})\cong\Gamma^{(n)},
\]
and identified the homology Schubert classes with the type C $k$-Schur
functions $P_w^{(n)}$, indexed by affine Grassmannian elements
$w\in\tilc^0$~\cite[Theorems~1.3 and~1.4]{LSS10c}. This identification
motivates the search for a tableau description of these functions.
Morse and Seelinger proposed such a description using strong marked
tableaux on symmetric $2n$-cores, as recorded
in~\cite[Section~5.2]{See}.

The following conjecture is the starting point of our investigation.
The bijection $\mathfrak{a}$ and the tableau set $\CSMT{\kappa}$ are
recalled in Sections~\ref{s:Def} and~\ref{s:SMT}, respectively.
The conjecture appears in~\cite[Conjecture~5.2.6]{See}.
\begin{conjecture}
\label{conj:conjP}
Let $w\in\tilc^0$ and $\kappa=\mathfrak{a}(w)$. Then
\begin{equation}
\label{eq:conjP}
P_w^{(n)}({\bf x})
=
\sum_{T\in\CSMT{\kappa}}{\bf x}^{{\rm wt}(T)}.
\end{equation}
\end{conjecture}

An approach to Conjecture~\ref{conj:conjP} requires a description of the
tableaux that retains both their underlying strong covers and their
markings. The type A theory suggests that an explicit permutation model
is useful for this purpose, since it expresses the relevant chains in
terms of reflections and integer labels~\cite{LLMSSZ, LLMS}. Our aim is
to develop this description in type C and to make the related action on
bounded partitions equally explicit. We do not prove
Conjecture~\ref{conj:conjP} here; rather, we provide combinatorial tools
for studying its tableau model.

\subsection{Type C transpositions and marked covers}

The constructions in this paper use the known bijections
\[
\tilc^0
\xrightarrow{\ \mathfrak{a}\ }
\Cco
\xrightarrow{\ \mathfrak{p}\ }
\Cpar,
\]
where $\Cco$ is the set of symmetric $2n$-cores and $\Cpar$ consists of
partitions with largest part at most $2n$ in which every part of size at
most $n$ occurs at most once; see~\cite{HJ,See} and
Lemmas~\ref{lem:bij1} and~\ref{lem:bij2}. These bijections identify the
underlying indexing sets, but an explicit description of marked covers
requires additional information. Indeed, the skew shape of a strong
cover may contain several ribbon components, and a marking specifies
the content of the southeast-most box of one of these components.
Thus, identifying the reflection relating the two affine Grassmannian
elements is not enough: one must also recover the chosen marking from
the permutation data.

Our first main contribution is a transposition description that retains
this marking. We work with the realization of $\tilc$ as permutations
of $\mathbb{Z}$ recalled from~\cite{Shi}. Edge sequences encode cores
using integer labels outside $(n+1)\mathbb{Z}$, and
Proposition~\ref{prop:twoaction} identifies the resulting action on cores
with the corner-addition and corner-removal action of
Lemma~\ref{lem:action}. We then introduce type C transpositions
$\trans{i}{j}$ adapted to the periodicity and symmetry of these
permutations. Theorem~\ref{thm:reflection} gives the conjugation identity
\[
w\trans{i}{j}w^{-1}
=
\trans{w(i)}{w(j)},
\]
for $w\in\tilc$ and the indices specified in the theorem. This identity
relates transpositions acting on positions to transpositions acting on
values.

The key connection with the tableau marking is established in
Theorem~\ref{thm:main}. Let
$C=(\SMC{\tau}{c}{\kappa})$ be a strong marked cover, with
$\tau=\mathfrak{a}(w)$ and $\kappa=\mathfrak{a}(u)$. The theorem gives
indices $i,j\in\mathbb{Z}\setminus(n+1)\mathbb{Z}$ such that
\[
u=w\trans{i}{j},
\qquad
w(j)=u(i)=c+\left\lfloor\frac{c}{n}\right\rfloor+1.
\]
The second equality is essential: it links the content of the selected
ribbon head to a specific value of the affine permutations, rather than
recording only the unmarked cover. Applying this description to
successive covers gives permutation formulations of strong marked
strips and tableaux, with the increasing-mark condition and the strip
sizes retained. The contribution therefore lies in making the
reflection and marking data explicit, beyond the identification of
cores with affine Grassmannian elements.

\subsection{The induced action on bounded partitions}
Our second main contribution concerns the simple-reflection action on
the partition model. The action of $\tilc$ on symmetric cores and the
bijection $\mathfrak{p}$ induce an action on $\Cpar$ by
\[
w\cdot\lambda
:=
\mathfrak{p}\bigl(w\cdot\mathfrak{p}^{-1}(\lambda)\bigr),
\qquad w\in\tilc,
\quad \lambda\in\Cpar.
\]
Although this formula defines the action abstractly, it does not show
which parts of $\lambda$ change. Moreover, a simple reflection can add
or remove several corners of the associated core at once, so its effect
on the bounded partition is not immediately apparent from the change
in the core diagram.

Theorem~\ref{thm:silambda} resolves this issue by giving a one-box rule.
For $\kappa=\mathfrak{p}^{-1}(\lambda)$ and $i\in[0,n]$, its conclusion
can be written on the partition side as
\[
s_i\cdot\lambda
=
\begin{cases}
\lambda+\epsilon_a,
& \text{if }\dg{\kappa}\subsetneq\dg{s_i\cdot\kappa},\\
\lambda-\epsilon_r,
& \text{if }\dg{\kappa}\supsetneq\dg{s_i\cdot\kappa},\\
\lambda,
& \text{otherwise}.
\end{cases}
\]
In the first case, $a$ is the largest row index $c$ of an addable corner
$(c,d)$ of $\kappa$ satisfying $d\geq c$ and having residue congruent to
$i$ or $2n-i$ modulo $2n$. In the second case, $r$ is defined by the
same rule using removable corners. Thus, the theorem specifies not only
that a single part changes by one, but also the row in which that change
occurs.

The proof combines the structure of corners in $2n$-strings with the
hook-length description of $\mathfrak{p}$. In particular,
Proposition~\ref{prop:stringhook} identifies the boxes whose hook
lengths cross the $2n$-bound under a simple reflection. This analysis
explains how a simultaneous change of several core corners results in
a change in just one row of the bounded partition. Together with the
marked-cover description, the one-box rule gives explicit ways to
handle both strong covers and simple-reflection actions within the
same permutation--core--partition framework.

\smallskip

\noindent{\bf Outline of the paper.}
Section~\ref{s:Def} recalls the affine-permutation realization of
$\tilc$, its action on symmetric cores, and the two indexing bijections
in Lemmas~\ref{lem:bij1} and~\ref{lem:bij2}.
Section~\ref{s:SMT} reviews the core model for strong marked tableaux
and introduces edge sequences. After proving the compatibility of the
two core actions in Proposition~\ref{prop:twoaction}, it develops type C
transpositions and establishes the conjugation identity in
Theorem~\ref{thm:reflection} and the marked-cover description in
Theorem~\ref{thm:main}.
Section~\ref{s:action} studies the corner strings and hook-length
changes, culminating in Proposition~\ref{prop:stringhook} and the
explicit partition action in Theorem~\ref{thm:silambda}.

\smallskip
\noindent{\bf Notation.}
Throughout the paper, $n\geq2$. For integers $a,b$, write
$[a,b]=\{a,a+1,\ldots,b\}$, with $[a,b]=\emptyset$ when $a>b$.
Let $\epsilon_a$ denote the finitely supported sequence with a $1$ in
position $a$ and zeros elsewhere. For a finite integer sequence
$\gamma$, write $|\gamma|=\sum_a\gamma_a$.

\section{Preliminaries}\mlabel{s:Def}
In this section, we collect the background material and fix the notation that will be used throughout the paper. We first recall the realization of the affine Weyl group $\widetilde{C}_n$ as a group of permutations of $\mathbb{Z}$, together with the symmetry properties of these permutations and the notion of affine Grassmannian elements. We then review the combinatorial models for affine Grassmannian elements in terms of symmetric $2n$-cores and the corresponding class of partitions. In particular, we recall the action of $\widetilde{C}_n$ on symmetric cores and the bijections
\[
\widetilde{C}_n^0
\longleftrightarrow
\Cco
\longleftrightarrow
\Cpar.
\]
These realizations provide the basic dictionary between affine permutations, cores, and partitions that will be used in the subsequent construction of type C strong marked tableaux and in the study of the induced action on partitions.

\subsection{Affine Weyl group of type C}
In this subsection, we review some facts of permutations in the affine Weyl group of type C studied in~\cite{Shi}.
The affine Weyl group $\tilc$ of type C has generators $\{s_0, s_1, \ldots, s_n \}$ and relations
\begin{equation}
\begin{aligned}
s_i^2=&\ 1, \quad \text{for all $i\in[0,n]$}, \\
s_i s_j =&\ s_j s_i, \quad \text{if $|i-j|>1$}, \\
s_i s_{i+1} s_i =&\ s_{i+1} s_i s_{i+1}, \quad \text{if $1\leq i \leq n-2$},\\
s_0 s_1 s_0 s_1 =&\ s_1 s_0 s_1 s_0, \\
s_{n-1} s_n s_{n-1} s_n =&\ s_n s_{n-1} s_n s_{n-1}.
\end{aligned}
\mlabel{eq:Cgene}
\end{equation}
Given a $w\in\tilc$, denote $\ell(w)$ be the {\em length} of $w$ if there is a shortest expression $w = s_{i_1}\cdots s_{i_{\ell(w)}}$ of $w$. We call such an expression $w = s_{i_1}\cdots s_{i_{\ell(w)}}$ a {\em reduced expression}.

Shi~\cite{Shi} showed that the generators $\{s_0, s_1, \ldots, s_n \}$ in $\tilc$ represent the permutations on $\mathbb{Z}$ as the following Lemma.

\begin{lemma}{\rm (\cite[\S1.4]{Shi})}
Let $s_i\in\tilc$ with $i\in[0,n]$ and $x\in \mathbb{Z}$.
\begin{enumerate}
\item
If $0< i < n$, then
\begin{equation*}
s_i(x) =
\left\{
\begin{array}{ll}
x, & \quad \text{{\rm if} $x \neq \pm i, \pm (i+1)\,{\rm mod}\, (2n+2)$},\\
x+1, & \quad \text{{\rm if} $x = i, -i-1\,{\rm mod}\, (2n+2)$},\\
x-1, & \quad \text{{\rm if} $x = i+1, -i\,{\rm mod}\, (2n+2)$}.
\end{array}
\right.
\end{equation*}
\mlabel{it:Cperm1}

\item If $i=0$, then
\begin{equation*}
s_0(x) =
\left\{
\begin{array}{ll}
x, & \quad \text{{\rm if} $x \neq \pm 1\,{\rm mod}\, (2n+2)$},\\
x+2, & \quad \text{{\rm if} $x = -1\,{\rm mod}\, (2n+2)$},\\
x-2, & \quad \text{{\rm if} $x = 1\,{\rm mod}\, (2n+2)$}.
\end{array}
\right.
\end{equation*}
\mlabel{it:Cperm2}

\item If $i=n$, then
\begin{equation*}
s_n(x) =
\left\{
\begin{array}{ll}
x, & \quad \text{{\rm if} $x \neq n,n+2\,{\rm mod}\, (2n+2)$},\\
x+2, & \quad \text{{\rm if} $x = n\,{\rm mod}\, (2n+2)$},\\
x-2, & \quad \text{{\rm if} $x = n+2\,{\rm mod}\, (2n+2)$}.
\end{array}
\right.
\end{equation*}
\mlabel{it:Cperm3}
\end{enumerate}
\mlabel{lem:Cperm}
\end{lemma}

Note that $s_0$ and $s_n$ are symmetric in the study of $\tilc$.
A more detailed description of permutations in $\tilc$ was provided by~\cite{Shi} as follows.

\begin{lemma}{\rm (\cite[\S1.5]{Shi})}
For any $w\in\tilc$ and $a, x \in\mathbb{Z}$,
\begin{equation}
w(x)+ w(2a(n+1)-x) = 2a(n+1)
\mlabel{eq:CpermDesSym}
\end{equation}
In particular,
\begin{align*}
  w(x) = -w(-x), \qquad  w(a(n+1)) = a(n+1).
\end{align*}
\mlabel{lem:CpermDes}
\end{lemma}

The Weyl group $C_n$ of type C, generated by $\{ s_1,\ldots,s_n \}$, is a subgroup of $\tilc$. Let $\tilc^0$ be the minimal length coset representatives of $\tilc/C_n$. That is, for $w\in\tilc^0$, either $w=\id$ or $s_0$ is the only generator such that $\ell(ws_0)<\ell(w)$. The $\tilc^0$ is called the set of {\em affine Grassmannian elements} of $\tilc$.

\begin{exam}
Let $n=2$. Then $$w=s_2s_1s_2s_0s_1s_0s_2s_1s_0\in\tilc^0,$$ and the results of permuting $\mathbb{Z}$ by $w$ is:
\begin{equation*}
\begin{tabular}{c c c c c c c c}
-2 & -1 & 0 & 1 & 2 & 3 & 4 & 5\\
\hline
\textcolor{red}{5} & \textcolor{red}{10} & 0 & \textcolor{red}{-10} & \textcolor{red}{-5} & 3 & \textcolor{red}{11} & \textcolor{red}{16}
\end{tabular}
\end{equation*}
\end{exam}

\subsection{Affine permutations of type C, cores, and partitions}\mlabel{ss:PerCorePar}
This subsection is used to review bijections between affine permutations of type C, cores, and partitions. These bijections were first stated in~\cite{LSS10c}, and generated to types B and D by~\cite{HJ}. Later, the combinatorial version of such bijections were given by~\cite{See}.

Recall some notation used in~\cite{BMPS19,LLMSSZ,LSS10c}.
\begin{enumerate}[label=(\roman*)]
\item A {\em partition} $\lambda:=(\lambda_1,\ldots,\lambda_{\ell(\lambda)})$ is a sequence of weakly decreasing positive integers, where $\ell(\lambda)$ is the number of the entries of $\lambda$. We adopt a general convention that $\lambda_i = 0 $ if $i>\ell(\lambda)$, and let ${\rm P}$ be the set of all partitions.

\item The {\em Young (or Ferrers) diagram} of a partition $\lambda$ is the set of boxes $$\dg{\lambda}:=\{ (r,c)\in \mathbb{Z}_{\geq1}\times \mathbb{Z}_{\geq1} \mid c\leq \lambda_r\},$$ drawn in French notation so that rows (resp. columns) are increasing from south to north (resp. west to east). We use the convention $\dg{0} = \emptyset$.

\item Each box $c$ in a given Young diagram has the {\em hook length} $\hl{c}$, which counts the number of boxes strictly above it in its column and weakly to its right in its row.

\item An {\em $r$-core} is a partition with no box has hook length $r$ in its Young diagram. Let $\co{r}$ be the set collecting all $r$-cores, and let
\begin{equation}
\Cco:= \{ \kappa\in\co{2n} \mid (i,j)\in\dg{\kappa}, \text{ if }\,(j,i)\in\dg{\kappa} \}
\mlabel{eq:core}
\end{equation}
be the set of symmetric $2n$-cores.

\item The {\em content} of a box $c:=(i,j)$ in a given Young diagram is $j-i$, and the {\em residue} of $c$ in this paper is the content of $c$ mod $2n$, denoted by ${\rm res}(c)$.

\item We say that a box $c:=(i,j)$ lies on the {\em main diagonal} if $i=j$.

\item A box $c$ is an {\em addable corner} (resp. {\em removable corner}) of a partition $\lambda$ if $c\notin\dg{\lambda}$ (resp. $c\in\dg{\lambda}$) and $\dg{\lambda}\sqcup \{c\}$ (resp. $\dg{\lambda}\setminus \{c\}$) is still a Young diagram of a partition.

\item For two partitions $\lambda$ and $\mu$ such that $\dg{\lambda}\subseteq\dg{\mu}$, define $\mu/\lambda:=\dg{\mu}/\dg{\lambda}$.

\end{enumerate}

We now review the first bijection between $\tilc^0$ and $\Cco$.

\begin{lemma}{\rm (\cite[Theorem~5.8]{HJ})}
There is a group action of $$\tilc \times \Cco \longrightarrow \Cco, \quad  (s_i, \kappa)\longmapsto s_i\cdot \kappa$$ such that
\begin{enumerate}
\item if $\kappa$ has at least one addable corner of residue $i$, then
$s_i\cdot\kappa$ is obtained from $\kappa$ by adding all addable corners
of residue $i$ or $2n-i$ modulo $2n$.
\mlabel{it:action1}

\item if $\kappa$ has at least one removable corner of residue $i$, then $s_i\cdot \kappa$ is obtained from $\kappa$ by removing all removable corners of residue $i$ and $2n-i$ modulo $2n$,
\mlabel{it:action2}

\item otherwise, $s_i\cdot \kappa  := \kappa$.
\mlabel{it:action3}

\end{enumerate}
\mlabel{lem:action}
\end{lemma}

\begin{prop}
For $i\in[0,n]$ and $\kappa\in\Cco$, the addable (resp. removable) corners of $\kappa$ with residue $i$ and $2n-i$ {\rm mod} $2n$ are symmetric with respect to the main diagonal of $\kappa$.
\mlabel{prop:symm}
\end{prop}

\begin{proof}
We first consider the addable case.
If $i = 0$ or $n$, then the statement is true by $ i = $ $2n-i$ {\rm mod} $2n$. Otherwise, suppose $i\in[1,n-1]$ and $c:=(a,b)$ is an addable corner of $\kappa$ with residue $i = b-a$ mod $2n$. By symmetry,
$$c':= (b,a)\in s_i\cdot \kappa/\kappa, \text{ with }\, {\rm res}(c') = a-b\,\text{ mod $2n$}\, = 2n-i,$$
which follows that $c'$ is an addable corner of $\kappa$ with residue $2n-i$ {\rm mod} $2n$.
The case of removable is analogous, since addable and removable are inverse processes.
\end{proof}

\begin{lemma}{\rm (\cite{HJ}, \cite[Lemma~5.1.17]{See})} There is bijection
\begin{align*}
\mathfrak{a}: \tilc^0   \longrightarrow \, \Cco,  \quad
w   \longmapsto \, w\cdot \emptyset \, .
\end{align*}
\mlabel{lem:bij1}
\end{lemma}

The following example illustrates explicitly how the bijection $\mathfrak{a}$ in Lemma~\mref{lem:bij1} is realized by successively applying the simple reflections to the empty core.

\begin{exam}
Let $n=2$ and consider $w=s_2s_1s_2s_0s_1s_0s_2s_1s_0\in\tilde{C}_2^0$. Then
$$\mathfrak{a}(w) = (7,5,5,3,3,1,1)\in\Cco,$$
via
\allowdisplaybreaks{
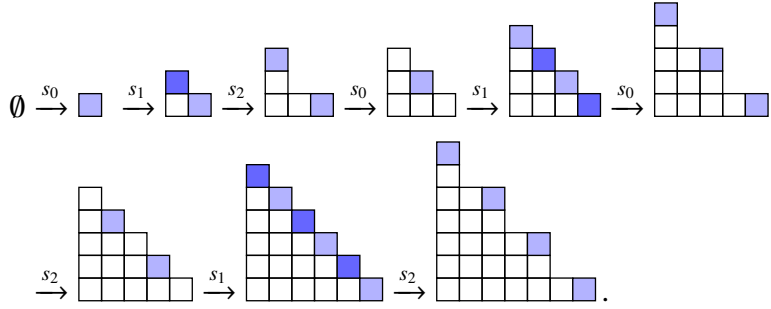
\begin{align*}
\emptyset & \xrightarrow{s_0}
\begin{tikzpicture}[scale=.3,line width=0.5pt,baseline=(a.base)]
\filldraw[blue!30,draw=black] (-1,-1) rectangle (0,0);
\node (a) [align=center] {\\[0pt] };
\end{tikzpicture}
\xrightarrow{s_1}
\begin{tikzpicture}[scale=.3,line width=0.5pt,baseline=(a.base)]
\draw[draw=black] (-1,-1) rectangle (0,0);
\filldraw[blue!30,draw=black] (0,-1) rectangle (1,0);
\filldraw[blue!60,draw=black] (-1,0) rectangle (0,1);
\node (a) [align=center] {\\[0pt] };
\end{tikzpicture}
\xrightarrow{s_2}
\begin{tikzpicture}[scale=.3,line width=0.5pt,baseline=(a.base)]
\draw[draw=black] (-1,-1) rectangle (0,0);
\draw[draw=black] (0,-1) rectangle (1,0);
\draw[draw=black] (-1,0) rectangle (0,1);
\filldraw[blue!30,draw=black] (1,-1) rectangle (2,0);
\filldraw[blue!30,draw=black] (-1,1) rectangle (0,2);
\node (a) [align=center] {\\[0pt] };
\end{tikzpicture}
\xrightarrow{s_0}
\begin{tikzpicture}[scale=.3,line width=0.5pt,baseline=(a.base)]
\draw[draw=black] (-1,-1) rectangle (0,0);
\draw[draw=black] (0,-1) rectangle (1,0);
\draw[draw=black] (-1,0) rectangle (0,1);
\draw[draw=black] (1,-1) rectangle (2,0);
\draw[draw=black] (-1,1) rectangle (0,2);
\filldraw[blue!30,draw=black] (0,0) rectangle (1,1);
\node (a) [align=center] {\\[0pt] };
\end{tikzpicture}
\xrightarrow{s_1}
\begin{tikzpicture}[scale=.3,line width=0.5pt,baseline=(a.base)]
\draw[draw=black] (-1,-1) rectangle (0,0);
\draw[draw=black] (0,-1) rectangle (1,0);
\draw[draw=black] (-1,0) rectangle (0,1);
\draw[draw=black] (1,-1) rectangle (2,0);
\draw[draw=black] (-1,1) rectangle (0,2);
\draw[draw=black] (0,0) rectangle (1,1);
\filldraw[blue!30,draw=black] (-1,2) rectangle (0,3);
\filldraw[blue!60,draw=black] (2,-1) rectangle (3,0);
\filldraw[blue!60,draw=black] (0,1) rectangle (1,2);
\filldraw[blue!30,draw=black] (1,0) rectangle (2,1);
\node (a) [align=center] {\\[0pt] };
\end{tikzpicture}
\xrightarrow{s_0}
\begin{tikzpicture}[scale=.3,line width=0.5pt,baseline=(a.base)]
\draw[draw=black] (-1,-1) rectangle (0,0);
\draw[draw=black] (0,-1) rectangle (1,0);
\draw[draw=black] (-1,0) rectangle (0,1);
\draw[draw=black] (1,-1) rectangle (2,0);
\draw[draw=black] (-1,1) rectangle (0,2);
\draw[draw=black] (0,0) rectangle (1,1);
\draw[draw=black] (-1,2) rectangle (0,3);
\draw[draw=black] (2,-1) rectangle (3,0);
\draw[draw=black] (0,1) rectangle (1,2);
\draw[draw=black] (1,0) rectangle (2,1);
\filldraw[blue!30,draw=black] (3,-1) rectangle (4,0);
\filldraw[blue!30,draw=black] (-1,3) rectangle (0,4);
\filldraw[blue!30,draw=black] (1,1) rectangle (2,2);
\node (a) [align=center] {\\[0pt] };
\end{tikzpicture}\\
& \xrightarrow{s_2}
\begin{tikzpicture}[scale=.3,line width=0.5pt,baseline=(a.base)]
\draw[draw=black] (-1,-1) rectangle (0,0);
\draw[draw=black] (0,-1) rectangle (1,0);
\draw[draw=black] (-1,0) rectangle (0,1);
\draw[draw=black] (1,-1) rectangle (2,0);
\draw[draw=black] (-1,1) rectangle (0,2);
\draw[draw=black] (0,0) rectangle (1,1);
\draw[draw=black] (-1,2) rectangle (0,3);
\draw[draw=black] (2,-1) rectangle (3,0);
\draw[draw=black] (0,1) rectangle (1,2);
\draw[draw=black] (1,0) rectangle (2,1);
\draw[draw=black] (3,-1) rectangle (4,0);
\draw[draw=black] (-1,3) rectangle (0,4);
\draw[draw=black] (1,1) rectangle (2,2);
\filldraw[blue!30,draw=black] (0,2) rectangle (1,3);
\filldraw[blue!30,draw=black] (2,0) rectangle (3,1);
\node (a) [align=center] {\\[0pt] };
\end{tikzpicture}
\xrightarrow{s_1}
\begin{tikzpicture}[scale=.3,line width=0.5pt,baseline=(a.base)]
\draw[draw=black] (-1,-1) rectangle (0,0);
\draw[draw=black] (0,-1) rectangle (1,0);
\draw[draw=black] (-1,0) rectangle (0,1);
\draw[draw=black] (1,-1) rectangle (2,0);
\draw[draw=black] (-1,1) rectangle (0,2);
\draw[draw=black] (0,0) rectangle (1,1);
\draw[draw=black] (-1,2) rectangle (0,3);
\draw[draw=black] (2,-1) rectangle (3,0);
\draw[draw=black] (0,1) rectangle (1,2);
\draw[draw=black] (1,0) rectangle (2,1);
\draw[draw=black] (3,-1) rectangle (4,0);
\draw[draw=black] (-1,3) rectangle (0,4);
\draw[draw=black] (1,1) rectangle (2,2);
\draw[draw=black] (0,2) rectangle (1,3);
\draw[draw=black] (2,0) rectangle (3,1);
\filldraw[blue!60,draw=black] (-1,4) rectangle (0,5);
\filldraw[blue!30,draw=black] (4,-1) rectangle (5,0);
\filldraw[blue!60,draw=black] (3,0) rectangle (4,1);
\filldraw[blue!30,draw=black] (0,3) rectangle (1,4);
\filldraw[blue!30,draw=black] (2,1) rectangle (3,2);
\filldraw[blue!60,draw=black] (1,2) rectangle (2,3);
\node (a) [align=center] {\\[0pt] };
\end{tikzpicture}
\xrightarrow{s_2}
\begin{tikzpicture}[scale=.3,line width=0.5pt,baseline=(a.base)]
\draw[draw=black] (-1,-1) rectangle (0,0);
\draw[draw=black] (0,-1) rectangle (1,0);
\draw[draw=black] (-1,0) rectangle (0,1);
\draw[draw=black] (1,-1) rectangle (2,0);
\draw[draw=black] (-1,1) rectangle (0,2);
\draw[draw=black] (0,0) rectangle (1,1);
\draw[draw=black] (-1,2) rectangle (0,3);
\draw[draw=black] (2,-1) rectangle (3,0);
\draw[draw=black] (0,1) rectangle (1,2);
\draw[draw=black] (1,0) rectangle (2,1);
\draw[draw=black] (3,-1) rectangle (4,0);
\draw[draw=black] (-1,3) rectangle (0,4);
\draw[draw=black] (1,1) rectangle (2,2);
\draw[draw=black] (0,2) rectangle (1,3);
\draw[draw=black] (2,0) rectangle (3,1);
\draw[draw=black] (-1,4) rectangle (0,5);
\draw[draw=black] (4,-1) rectangle (5,0);
\draw[draw=black] (3,0) rectangle (4,1);
\draw[draw=black] (0,3) rectangle (1,4);
\draw[draw=black] (2,1) rectangle (3,2);
\draw[draw=black] (1,2) rectangle (2,3);
\filldraw[blue!30,draw=black] (5,-1) rectangle (6,0);
\filldraw[blue!30,draw=black] (-1,5) rectangle (0,6);
\filldraw[blue!30,draw=black] (3,1) rectangle (4,2);
\filldraw[blue!30,draw=black] (1,3) rectangle (2,4);
\node (a) [align=center] {\\[0pt] };
\end{tikzpicture}
\, .
\end{align*}
}
Here the light blue boxes are the whole addable corners of residue $i$, and the dark blue boxes are the whole addable corners of residue $2n-i$ with $i\neq 0,n$, when $s_i$ acts on $\Cco$.
\mlabel{ex:bij1}
\end{exam}

Now we review the bijection between $\Cco$ and partitions. In the study of affine Schubert calculus for $\sym$~\cite{LSS10c}, the subscripts of Schubert classes are a subset of partitions defined as
\begin{equation*}
\Cpar
:=
\left\{
\lambda\in{\rm P}
\;\middle|\;
\lambda_1\leq 2n,\
\lambda_i>\lambda_{i+1}
\text{ whenever }
i<\ell(\lambda)\text{ and }\lambda_i\leq n
\right\}.
\end{equation*}

\begin{lemma}{\rm (\cite[Lemma~5.1.19]{See})}
There is a bijection $$\mathfrak{p}:\Cco\to\Cpar, \quad \kappa \longmapsto \mathfrak{p}(\kappa),$$  with  $\mathfrak{p}(\kappa)$ defined via the following procedure:
\begin{enumerate}
\item denote
\begin{equation*}
{\rm skew}(\kappa) := \{ c\in\dg{\kappa} \mid \hl{c}<2n \},
\end{equation*}

\item let
\begin{align*}
{\rm skew}^{\perp}(\kappa):=&\ \{ c:=(i,j)\in {\rm skew}(\kappa) \mid i<j \},  \\
{\rm skew}^{\perp}_a(\kappa):=&\ \{ c:=(i,j)\in {\rm skew}^{\perp}(\kappa)\mid i=a \},
\end{align*}
for some $a\in\mathbb{Z}_{\geq1}$,

\item $\mathfrak{p}(\kappa):= (\lambda_1, \ldots, \lambda_\ell)$, where
\begin{align*}
\ell:=&\ \#\{ c:= (i,j)\in \dg{\kappa} \mid i=j\},  \\
\lambda_a :=&\ 1+ \#{\rm skew}^{\perp}_a(\kappa), \,\text{ for }\, a\in[1,\ell].
\end{align*}
\end{enumerate}
\mlabel{lem:bij2}
\end{lemma}

The following example illustrates the construction of the bijection $\mathfrak{p}$ in Lemma~\mref{lem:bij2} by determining the corresponding partition from the hook lengths and the associated sets ${\rm skew}(\kappa)$ and ${\rm skew}^{\perp}(\kappa)$.

\begin{exam}
Let $n=2$ and $\kappa:=(7,5,5,3,3,1,1) \in\Cco$.
Then $\mathfrak{p}(\kappa) = (3,3,3)\in\Cpar$ by
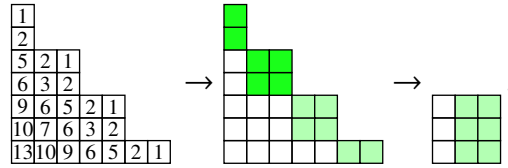
\begin{equation*}
\begin{tikzpicture}[scale=.3,line width=0.5pt,baseline=(a.base)]
\draw[draw=black] (-1,-4) rectangle (0,-3);\node at(-0.5,-3.5){\tiny\( 13 \)};
\draw[draw=black] (0,-4) rectangle (1,-3);\node at(0.5,-3.5){\tiny\( 10 \)};
\draw[draw=black] (-1,-3) rectangle (0,-2);\node at(-0.5,-2.5){\tiny\( 10 \)};
\draw[draw=black] (1,-4) rectangle (2,-3);\node at(1.5,-3.5){\tiny\( 9 \)};
\draw[draw=black] (-1,-2) rectangle (0,-1);\node at(-0.5,-1.5){\tiny\( 9 \)};
\draw[draw=black] (0,-3) rectangle (1,-2);\node at(0.5,-2.5){\tiny\( 7 \)};
\draw[draw=black] (-1,-1) rectangle (0,0);\node at(-0.5,-0.5){\tiny\( 6 \)};
\draw[draw=black] (2,-4) rectangle (3,-3);\node at(2.5,-3.5){\tiny\( 6 \)};
\draw[draw=black] (0,-2) rectangle (1,-1);\node at(0.5,-1.5){\tiny\( 6 \)};
\draw[draw=black] (1,-3) rectangle (2,-2);\node at(1.5,-2.5){\tiny\( 6 \)};
\draw[draw=black] (3,-4) rectangle (4,-3);\node at(3.5,-3.5){\tiny\( 5 \)};
\draw[draw=black] (-1,0) rectangle (0,1);\node at(-0.5,0.5){\tiny\( 5 \)};
\draw[draw=black] (1,-2) rectangle (2,-1);\node at(1.5,-1.5){\tiny\( 5 \)};
\draw[draw=black] (0,-1) rectangle (1,0);\node at(0.5,-0.5){\tiny\( 3 \)};
\draw[draw=black] (2,-3) rectangle (3,-2);\node at(2.5,-2.5){\tiny\( 3 \)};
\draw[draw=black] (-1,1) rectangle (0,2);\node at(-0.5,1.5){\tiny\( 2 \)};
\draw[draw=black] (4,-4) rectangle (5,-3);\node at(4.5,-3.5){\tiny\( 2 \)};
\draw[draw=black] (3,-3) rectangle (4,-2);\node at(3.5,-2.5){\tiny\( 2 \)};
\draw[draw=black] (0,0) rectangle (1,1);\node at(0.5,0.5){\tiny\( 2 \)};
\draw[draw=black] (2,-2) rectangle (3,-1);\node at(2.5,-1.5){\tiny\( 2 \)};
\draw[draw=black] (1,-1) rectangle (2,0);\node at(1.5,-0.5){\tiny\( 2 \)};
\draw[draw=black] (5,-4) rectangle (6,-3);\node at(5.5,-3.5){\tiny\( 1 \)};
\draw[draw=black] (-1,2) rectangle (0,3);\node at(-0.5,2.5){\tiny\( 1 \)};
\draw[draw=black] (3,-2) rectangle (4,-1);\node at(3.5,-1.5){\tiny\( 1 \)};
\draw[draw=black] (1,0) rectangle (2,1);\node at(1.5,0.5){\tiny\( 1 \)};
\node (a) [align=center] {\\[0pt] };
\end{tikzpicture}
\to
\begin{tikzpicture}[scale=.3,line width=0.5pt,baseline=(a.base)]
\draw[draw=black] (-1,-4) rectangle (0,-3);
\draw[draw=black] (0,-4) rectangle (1,-3);
\draw[draw=black] (-1,-3) rectangle (0,-2);
\draw[draw=black] (1,-4) rectangle (2,-3);
\draw[draw=black] (-1,-2) rectangle (0,-1);
\draw[draw=black] (0,-3) rectangle (1,-2);
\draw[draw=black] (-1,-1) rectangle (0,0);
\draw[draw=black] (2,-4) rectangle (3,-3);
\draw[draw=black] (0,-2) rectangle (1,-1);
\draw[draw=black] (1,-3) rectangle (2,-2);
\draw[draw=black] (3,-4) rectangle (4,-3);
\draw[draw=black] (-1,0) rectangle (0,1);
\draw[draw=black] (1,-2) rectangle (2,-1);
\filldraw[green!90,draw=black] (0,-1) rectangle (1,0);
\filldraw[green!30,draw=black] (2,-3) rectangle (3,-2);
\filldraw[green!90,draw=black] (-1,1) rectangle (0,2);
\filldraw[green!30,draw=black] (4,-4) rectangle (5,-3);
\filldraw[green!30,draw=black] (3,-3) rectangle (4,-2);
\filldraw[green!90,draw=black] (0,0) rectangle (1,1);
\filldraw[green!30,draw=black] (2,-2) rectangle (3,-1);
\filldraw[green!90,draw=black] (1,-1) rectangle (2,0);
\filldraw[green!30,draw=black] (5,-4) rectangle (6,-3);
\filldraw[green!90,draw=black] (-1,2) rectangle (0,3);
\filldraw[green!30,draw=black] (3,-2) rectangle (4,-1);
\filldraw[green!90,draw=black] (1,0) rectangle (2,1);
\node (a) [align=center] {\\[0pt] };
\end{tikzpicture}
\to
\begin{tikzpicture}[scale=.3,line width=0.5pt,baseline=(a.base)]
\draw[draw=black] (-1,-4) rectangle (0,-3);
\filldraw[green!30,draw=black] (0,-4) rectangle (1,-3);
\filldraw[green!30,draw=black] (1,-4) rectangle (2,-3);
\draw[draw=black] (-1,-3) rectangle (0,-2);
\filldraw[green!30,draw=black] (0,-3) rectangle (1,-2);
\filldraw[green!30,draw=black] (1,-3) rectangle (2,-2);
\draw[draw=black] (-1,-2) rectangle (0,-1);
\filldraw[green!30,draw=black] (0,-2) rectangle (1,-1);
\filldraw[green!30,draw=black] (1,-2) rectangle (2,-1);
\node (a) [align=center] {\\[0pt] };
\end{tikzpicture}
\, ,
\end{equation*}
where each number is the hook length of the corresponding box in $\dg{\kappa}$, all the green boxes are ${\rm skew}(\kappa)$, and the light green boxes are ${\rm skew}^\perp(\kappa)$.
\end{exam}

\section{A permutation version of type C strong marked tableaux}\mlabel{s:SMT}

Motivated by Conjecture~\ref{conj:conjP}, in this section we develop a
permutation model for type C strong marked tableaux. Starting from the
core-theoretic framework of~\cite[\S5.2.1]{See}, we use edge sequences to
translate strong marked covers, strips, and tableaux into the language of
affine type C permutations, following the general philosophy of~\cite{LLMS}.


\subsection{The core model for strong marked tableaux}

We begin by fixing the core-theoretic terminology and notation needed for
the permutation construction. In particular, we recall strong covers,
their ribbon components, strong marked covers, strong marked strips, and
strong marked tableaux for symmetric $2n$-cores.

\begin{defn}{\rm (\cite[Definition~5.2.1]{See})}
For $\tau$ and $\kappa$ in $\Cco$,  we write $\tau\Rightarrow\kappa$ if
$\dg{\tau}\subset\dg{\kappa}$ and
\[
\ell\bigl(\mathfrak{a}^{-1}(\kappa)\bigr)
=
\ell\bigl(\mathfrak{a}^{-1}(\tau)\bigr)+1.
\]
In this case, we say that $\kappa$ {\em strongly covers} $\tau$, or equivalently that $\tau\Rightarrow\kappa$ is a {\em strong cover}.
\mlabel{defn:SC}
\end{defn}

The idea of the strong cover comes from the following partial order, named the strong (Bruhat) order on $\tilc$. Owing to the bijection between $\tilc^0$ and $\Cco$ in Lemma~\ref{lem:bij1}, the strong cover is essentially the minimal strong order of core version.

\begin{defn}{\rm (\cite[\S3.1]{LSS10c})}
The {\em strong (Bruhat) order} on $\tilc$ is define by, for $w$ and $u$ in $\tilc$, $w\leq u$ if some (equivalently every) reduced expression of $u$ has a subword that is a reduced expression for $w$.
\mlabel{defn:SBO}
\end{defn}

The following lemma collects two fundamental properties of strong covers and the strong order that will be used repeatedly in the sequel.
A set of boxes in a Young diagram is called {\em connected}, if for any two boxes, there is a sequence of boxes in the set from one to the other where consecutive boxes share an edge. A connected set of boxes is called a {\em ribbon} if it does not contain any $2\times2$ subset of boxes. Notice that a partition $\kappa$ is a $2n$-core if and only if $\kappa$ has no removable ribbon consisting of $2n$ boxes.

\begin{lemma}
\begin{enumerate}
\item (\cite[Theorem~5.11]{HJ}) Let $v,w\in\tilc^0$ and $\tau,\kappa\in\Cco$ such that $\tau = \mathfrak{a}(v)$ and $\kappa = \mathfrak{a}(w)$. Then $v\leq w$ if and only if $\dg{\tau} \subseteq \dg{\kappa}$.
\mlabel{it:SCSBO1}

\item (\cite[Proposition~5.2.2~(b)]{See}) If $\tau\Rightarrow\kappa$, then $\kappa/\tau$ is a collection of disjoint ribbons.
\end{enumerate}
\mlabel{lem:SCSBO}
\end{lemma}

\begin{defn}{\rm (\cite[Definition~5.2.3]{See})}
\begin{enumerate}
\item For $\tau$ and $\kappa$ in $\Cco$, a {\em strong marked cover} $$C:=(\SMC{\tau}{c}{\kappa})$$
    is defined to be a strong cover $\tau\Rightarrow\kappa$ together with a marking $c\in\mathbb{Z}$, where $c$ is the content of the most southeast box in one of the ribbon in $\kappa/\tau$. We use the notation {\rm inside}$(C):=\tau$, {\rm outside}$(C):=\kappa$ and $m(C):=c$.
\mlabel{it:SMC1}

\item A {\em strong marked strip} $S$ of size $r\in [0,2n]$ is a sequence of strong marked covers
    \[
    S:= (C_1,\ldots,C_r),
    \]
    such that {\rm outside}$(C_k) = $ {\rm inside}$(C_{k+1})$ for each $k\in[1,r-1]$, and $m(C_1)< \cdots <m(C_r)$. Denote {\rm inside}$(S):= $ {\rm inside}$(C_1)$, {\rm outside}$(S):= $ {\rm outside}$(C_r)$ and {\rm size}$(S):=r$. In particular, if $S$is empty, we say that ${\rm size}(S) = 0$.
\mlabel{it:SMC2}

\item Let $\tau$ and $\kappa$ be in $\Cco$ with $\dg{\tau}\subseteq\dg{\kappa}$. A {\em strong marked tableau} of shape $\kappa/\tau$ and weight $(r_1,\ldots,r_s)$ is a sequence of strong marked strips
    \[
    T:=(S_1,\ldots,S_s),
    \]
    such that {\rm inside}$(S_1) = \tau$, ${\rm outside}(S_i) = {\rm inside}(S_{i+1})$ for $i\in[1,s-1]$, {\rm outside}$(S_s) = \kappa$, and {\rm size}$(S_i) = r_i$ for all $i\in[1,s]$.
    If $S_i$ is empty for some $i\in[1,s]$, we use the convention that
    $$
    {\rm outside}(S_{i-1})=:{\rm inside}(S_i) = {\rm outside}(S_i):=  {\rm inside}(S_{i+1}),
    $$
    where ${\rm outside}(S_{0}):=\tau$ and ${\rm inside}(S_{s+1}):=\kappa$.
    Denote {\rm wt}$(T):=(r_1,\ldots,r_s)$.
    In particular, if $\dg{\tau}=\emptyset$, then we say that $T$ has shape $\kappa$.
    Let $\CSMT{\kappa}$ be the set of all strong marked tableaux of shape $\kappa$.
\mlabel{it:SMC3}
\end{enumerate}
\mlabel{defn:SMC}
\end{defn}

\begin{exam}
Let $n=2$ and
\begin{align*}
\tau:=&\ (4,3,2,1) = \mathfrak{a}(s_1s_0s_2s_1s_0)\in\Cco,\\
\kappa:=&\  (7,5,5,3,3,1,1) = \mathfrak{a}(s_2s_1s_2s_0s_1s_0s_2s_1s_0)\in\Cco.
\end{align*}
Notice that $\dg{\tau}\subseteq\dg{\kappa}$. Then there is a strong marked tableau of shape $\kappa/\tau$
\begin{equation*}
T =
\begin{tikzpicture}[scale=.3,line width=0.5pt,baseline=(a.base)]
\draw[draw=black] (-1,-1) rectangle (0,0);
\draw[draw=black] (0,-1) rectangle (1,0);
\draw[draw=black] (-1,0) rectangle (0,1);
\draw[draw=black] (1,-1) rectangle (2,0);
\draw[draw=black] (-1,1) rectangle (0,2);
\draw[draw=black] (0,0) rectangle (1,1);
\draw[draw=black] (-1,2) rectangle (0,3);
\draw[draw=black] (2,-1) rectangle (3,0);
\draw[draw=black] (0,1) rectangle (1,2);
\draw[draw=black] (1,0) rectangle (2,1);
\node (a) [align=center] {\\[0pt] };
\draw[line width=0.8pt, black]
        (-1.5,-1) --  (-1.5, 6) arc(180:140:2)  (-1.5,-1) arc(180:220:2) ;
\end{tikzpicture}
\overset{-4}{\Rightarrow}
\begin{tikzpicture}[scale=.3,line width=0.5pt,baseline=(a.base)]
\draw[draw=black] (-1,-1) rectangle (0,0);
\draw[draw=black] (0,-1) rectangle (1,0);
\draw[draw=black] (-1,0) rectangle (0,1);
\draw[draw=black] (1,-1) rectangle (2,0);
\draw[draw=black] (-1,1) rectangle (0,2);
\draw[draw=black] (0,0) rectangle (1,1);
\draw[draw=black] (-1,2) rectangle (0,3);
\draw[draw=black] (2,-1) rectangle (3,0);
\draw[draw=black] (0,1) rectangle (1,2);
\draw[draw=black] (1,0) rectangle (2,1);
\filldraw[yellow!90,draw=black] (3,-1) rectangle (4,0);
\filldraw[yellow!90,draw=black] (-1,3) rectangle (0,4);\node at(-0.5,3.5){\tiny\( \star \)};
\filldraw[yellow!90,draw=black] (1,1) rectangle (2,2);
\node (a) [align=center] {\\[0pt] };
\end{tikzpicture}
\overset{5}{\Rightarrow}
\begin{tikzpicture}[scale=.3,line width=0.5pt,baseline=(a.base)]
\draw[draw=black] (-1,-1) rectangle (0,0);
\draw[draw=black] (0,-1) rectangle (1,0);
\draw[draw=black] (-1,0) rectangle (0,1);
\draw[draw=black] (1,-1) rectangle (2,0);
\draw[draw=black] (-1,1) rectangle (0,2);
\draw[draw=black] (0,0) rectangle (1,1);
\draw[draw=black] (-1,2) rectangle (0,3);
\draw[draw=black] (2,-1) rectangle (3,0);
\draw[draw=black] (0,1) rectangle (1,2);
\draw[draw=black] (1,0) rectangle (2,1);
\draw[draw=black] (3,-1) rectangle (4,0);
\draw[draw=black] (-1,3) rectangle (0,4);
\draw[draw=black] (1,1) rectangle (2,2);
\filldraw[yellow!90,draw=black] (4,-1) rectangle (5,0);\node at(4.5,-0.5){\tiny\( \star \)};
\filldraw[yellow!90,draw=black] (-1,4) rectangle (0,5);
\node (a) [align=center] {\\[0pt] };
\end{tikzpicture}
\overset{6}{\Rightarrow}
\begin{tikzpicture}[scale=.3,line width=0.5pt,baseline=(a.base)]
\draw[draw=black] (-1,-1) rectangle (0,0);
\draw[draw=black] (0,-1) rectangle (1,0);
\draw[draw=black] (-1,0) rectangle (0,1);
\draw[draw=black] (1,-1) rectangle (2,0);
\draw[draw=black] (-1,1) rectangle (0,2);
\draw[draw=black] (0,0) rectangle (1,1);
\draw[draw=black] (-1,2) rectangle (0,3);
\draw[draw=black] (2,-1) rectangle (3,0);
\draw[draw=black] (0,1) rectangle (1,2);
\draw[draw=black] (1,0) rectangle (2,1);
\draw[draw=black] (3,-1) rectangle (4,0);
\draw[draw=black] (-1,3) rectangle (0,4);
\draw[draw=black] (1,1) rectangle (2,2);
\draw[draw=black] (4,-1) rectangle (5,0);
\draw[draw=black] (-1,4) rectangle (0,5);
\filldraw[yellow!90,draw=black] (5,-1) rectangle (6,0);\node at(5.5,-0.5){\tiny\( \star \)};
\filldraw[yellow!90,draw=black] (-1,5) rectangle (0,6);
\filldraw[yellow!90,draw=black] (2,0) rectangle (3,1);
\filldraw[yellow!90,draw=black] (0,2) rectangle (1,3);
\node (a) [align=center] {\\[0pt] };
\draw[line width=0.8pt, black]
        (6.5,-1) --  (6.5, 6) arc(0:40:2)  (6.5,-1) arc(0:-40:2) ;
\end{tikzpicture}
\, , \,
\begin{tikzpicture}[scale=.3,line width=0.5pt,baseline=(a.base)]
\draw[draw=black] (-1,-1) rectangle (0,0);
\draw[draw=black] (0,-1) rectangle (1,0);
\draw[draw=black] (-1,0) rectangle (0,1);
\draw[draw=black] (1,-1) rectangle (2,0);
\draw[draw=black] (-1,1) rectangle (0,2);
\draw[draw=black] (0,0) rectangle (1,1);
\draw[draw=black] (-1,2) rectangle (0,3);
\draw[draw=black] (2,-1) rectangle (3,0);
\draw[draw=black] (0,1) rectangle (1,2);
\draw[draw=black] (1,0) rectangle (2,1);
\draw[draw=black] (3,-1) rectangle (4,0);
\draw[draw=black] (-1,3) rectangle (0,4);
\draw[draw=black] (1,1) rectangle (2,2);
\draw[draw=black] (4,-1) rectangle (5,0);
\draw[draw=black] (-1,4) rectangle (0,5);
\draw[draw=black] (5,-1) rectangle (6,0);
\draw[draw=black] (-1,5) rectangle (0,6);
\draw[draw=black] (2,0) rectangle (3,1);
\draw[draw=black] (0,2) rectangle (1,3);
\node (a) [align=center] {\\[0pt] };
\draw[line width=0.8pt, black]
        (-1.5,-1) --  (-1.5, 6) arc(180:140:2)  (-1.5,-1) arc(180:220:2) ;
\end{tikzpicture}
\overset{-1}{\Rightarrow}
\begin{tikzpicture}[scale=.3,line width=0.5pt,baseline=(a.base)]
\draw[draw=black] (-1,-1) rectangle (0,0);
\draw[draw=black] (0,-1) rectangle (1,0);
\draw[draw=black] (-1,0) rectangle (0,1);
\draw[draw=black] (1,-1) rectangle (2,0);
\draw[draw=black] (-1,1) rectangle (0,2);
\draw[draw=black] (0,0) rectangle (1,1);
\draw[draw=black] (-1,2) rectangle (0,3);
\draw[draw=black] (2,-1) rectangle (3,0);
\draw[draw=black] (0,1) rectangle (1,2);
\draw[draw=black] (1,0) rectangle (2,1);
\draw[draw=black] (3,-1) rectangle (4,0);
\draw[draw=black] (-1,3) rectangle (0,4);
\draw[draw=black] (1,1) rectangle (2,2);
\draw[draw=black] (0,2) rectangle (1,3);
\draw[draw=black] (2,0) rectangle (3,1);
\draw[draw=black] (-1,4) rectangle (0,5);
\draw[draw=black] (4,-1) rectangle (5,0);
\filldraw[yellow!90,draw=black] (3,0) rectangle (4,1);
\filldraw[yellow!90,draw=black] (0,3) rectangle (1,4);
\filldraw[yellow!90,draw=black] (2,1) rectangle (3,2);
\filldraw[yellow!90,draw=black] (1,2) rectangle (2,3);\node at(1.5,2.5){\tiny\( \star \)};
\draw[draw=black] (5,-1) rectangle (6,0);
\draw[draw=black] (-1,5) rectangle (0,6);
\filldraw[yellow!90,draw=black] (3,1) rectangle (4,2);
\filldraw[yellow!90,draw=black] (1,3) rectangle (2,4);
\node (a) [align=center] {\\[0pt] };
\draw[line width=0.8pt, black]
        (6.5,-1) --  (6.5, 6) arc(0:40:2)  (6.5,-1) arc(0:-40:2) ;
\end{tikzpicture}
\,
\end{equation*}
with {\rm wt}$(T) = (3,1)$. We encode this as
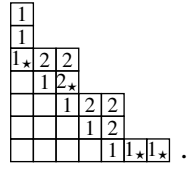
\[
\begin{tikzpicture}[scale=.3,line width=0.5pt,baseline=(a.base)]
\draw[draw=black] (-1,-1) rectangle (0,0);
\draw[draw=black] (0,-1) rectangle (1,0);
\draw[draw=black] (-1,0) rectangle (0,1);
\draw[draw=black] (1,-1) rectangle (2,0);
\draw[draw=black] (-1,1) rectangle (0,2);
\draw[draw=black] (0,0) rectangle (1,1);
\draw[draw=black] (-1,2) rectangle (0,3);
\draw[draw=black] (2,-1) rectangle (3,0);
\draw[draw=black] (0,1) rectangle (1,2);
\draw[draw=black] (1,0) rectangle (2,1);
\draw[draw=black] (3,-1) rectangle (4,0);\node at(3.5,-0.5){\tiny\( 1 \)};
\draw[draw=black] (-1,3) rectangle (0,4);\node at(-0.5,3.5){\tiny\( 1_\star \)};
\draw[draw=black] (1,1) rectangle (2,2);\node at(1.5,1.5){\tiny\( 1 \)};
\draw[draw=black] (0,2) rectangle (1,3);\node at(0.5,2.5){\tiny\( 1 \)};
\draw[draw=black] (2,0) rectangle (3,1);\node at(2.5,0.5){\tiny\( 1 \)};
\draw[draw=black] (-1,4) rectangle (0,5);\node at(-0.5,4.5){\tiny\( 1 \)};
\draw[draw=black] (4,-1) rectangle (5,0);\node at(4.5,-0.5){\tiny\( 1_\star \)};
\draw[draw=black] (3,0) rectangle (4,1);\node at(3.5,0.5){\tiny\( 2 \)};
\draw[draw=black] (0,3) rectangle (1,4);\node at(0.5,3.5){\tiny\( 2 \)};
\draw[draw=black] (2,1) rectangle (3,2);\node at(2.5,1.5){\tiny\( 2 \)};
\draw[draw=black] (1,2) rectangle (2,3);\node at(1.5,2.5){\tiny\( 2_\star \)};
\draw[draw=black] (5,-1) rectangle (6,0);\node at(5.5,-0.5){\tiny\( 1_\star \)};
\draw[draw=black] (-1,5) rectangle (0,6);\node at(-0.5,5.5){\tiny\( 1 \)};
\draw[draw=black] (3,1) rectangle (4,2);\node at(3.5,1.5){\tiny\( 2 \)};
\draw[draw=black] (1,3) rectangle (2,4);\node at(1.5,3.5){\tiny\( 2 \)};
\node (a) [align=center] {\\[0pt] };
\end{tikzpicture}.
\]
\mlabel{ex:SMC}
\end{exam}

\subsection{Edge sequence}
To pass from the core model to the permutation model, we use edge sequences to encode symmetric $2n$-cores in a form that is compatible with the action of the affine Weyl group.

We introduce below a slight variation, denoted by $p$, of the bijection constructed in~\mcite{LLMS}.
Consider a set of bi-infinite binary sequences:
\[
{\rm Eg}:=\big\{ (\ldots p_i \ldots)\mid i\in\mathbb{Z}\setminus\{ q(n+1),q\in\mathbb{Z}\}, p_i = 0\,\text{ or }\,1 \} \big\}.
\]
There is a bijection $$p:{\rm P}\to {\rm Eg}, \quad \lambda \mapsto p(\lambda)$$
Here $p(\lambda)$ is defined by traceing the border of $\dg{\lambda}$, going from northwest to southeast, such that every letter $1$ (resp. $0$) represents a south (resp. east) step, and some box with content $c$ is touched by the east step $p_x$ and south step $p_{c+v}$, where
\[
v:=\lfloor c/n \rfloor+1, \qquad
x:= \max{\big\{ x\in\mathbb{Z}\setminus\{ q(n+1),q\in\mathbb{Z}\} \mid x<c+v \big\}}.
\]
The $p(\lambda)$ is named an {\em edge sequence}.
We give an example to further elucidate the bijection $p$.

\begin{exam}
Let $n=2$.
Consider the diagram of the partition $\kappa:=(5,4,3,2,1)$ with its edge sequence $p(\kappa)$ labeled, where the $\bullet$ separates $p_{-1}$ and $p_1$:
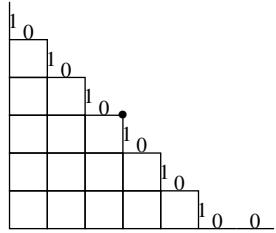
\begin{equation*}
\begin{tikzpicture}[scale=.5,line width=0.5pt,baseline=(a.base)]
\draw[draw=black] (-1,-1) rectangle (0,0);
\draw[draw=black] (0,-1) rectangle (1,0);
\draw[draw=black] (-1,0) rectangle (0,1);
\draw[draw=black] (1,-1) rectangle (2,0);
\draw[draw=black] (-1,1) rectangle (0,2);
\draw[draw=black] (0,0) rectangle (1,1);
\draw[draw=black] (-1,2) rectangle (0,3);
\draw[draw=black] (2,-1) rectangle (3,0);
\draw[draw=black] (0,1) rectangle (1,2);
\draw[draw=black] (1,0) rectangle (2,1);
\draw[draw=black] (3,-1) rectangle (4,0);
\draw[draw=black] (-1,3) rectangle (0,4);\node at(-0.5,4.2){\tiny\( 0 \)};
\node at(0.1,3.5){\tiny\( 1 \)};
\draw[draw=black] (1,1) rectangle (2,2);
\draw[draw=black] (0,2) rectangle (1,3);\node at(0.5,3.2){\tiny\( 0 \)};
\node at(1.1,2.5){\tiny\( 1 \)};
\node at(1.5,2.2){\tiny\( 0 \)};\node at(2.1,1.5){\tiny\( 1 \)};
\node at(2.5,1.2){\tiny\( 0 \)};\node at(3.1,0.5){\tiny\( 1 \)};
\node at(3.5,0.2){\tiny\( 0 \)};\node at(4.1,-0.5){\tiny\( 1 \)};
\node at(4.5,-0.8){\tiny\( 0 \)};
\node at(5.5,-0.8){\tiny\( 0 \)};
\draw[draw=black] (2,0) rectangle (3,1);
\node (a) [align=center] {\\[0pt] };
\node at(2,2){\tiny\( \bullet \)};
\draw[black,line width=0.5pt] (-1,4)--(-1,5);\node at(-0.9,4.5){\tiny\( 1 \)};
\draw[black,line width=0.5pt] (4,-1)--(5,-1);
\draw[black,line width=0.5pt] (5,-1)--(6,-1);
\end{tikzpicture}
\end{equation*}
\[
\begin{tabular}{c | c c c c c c c c c c c c c c c c c c c}
i & -8 & -7 & $\widehat{-6}$ & -5 & -4 & $\widehat{-3}$ & -2 & -1 & $\widehat{0}$ & 1 & 2 & $\widehat{3}$ & 4 & 5 & $\widehat{6}$ & 7 & 8 & $\widehat{9}$ & 10\\
\hline
$p_i$ & 1 & 0 & & 1 & 0 & & 1 & 0 & & 1 & 0 & & 1 & 0 & & 1 & 0 & & 0
\end{tabular}
\]
\mlabel{ex:edgeseq}
\end{exam}

There is a natural group action
\begin{align*}
*: \tilc\times {\rm P}  \longrightarrow \, {\rm P}, \quad
(s_i, \lambda)  \longmapsto \, p^{-1}\Big(s_i\big( p(\lambda)\big)\Big) \, .
\end{align*}
where $s_i(p_j):= p_{s_i(j)}$ for each $i\in[0,n]$ and $j\in\mathbb{Z}\setminus\{ q(n+1),q\in\mathbb{Z}\}$. We end this subsection by  showing that, when restricting {\rm P} to $\Cco$, the action $*$ is exactly the action $\cdot$ defined in Lemma~\ref{lem:action}.

\begin{prop}
For $i\in[0,n]$ and $\kappa\in\Cco$, we have
\begin{equation}
s_i\cdot \kappa = s_i * \kappa.
\mlabel{eq:twoaction}
\end{equation}
\mlabel{prop:twoaction}
\end{prop}

\begin{proof}
Let $\eta:=s_i\cdot \kappa$. It suffice to prove $s_i * \kappa = \eta$, which is divided into three cases due to Lemma~\ref{lem:action}.

\noindent{\bf Case~1.} $\kappa$ has an addable corner $c:= (a,b)$ of residue $i \equiv b-a$ mod $2n$. Then $\eta$ is $\kappa$ with all addable corners of residue $i$ and $2n-i$ {\rm mod} $2n$ added.

\noindent{\bf Subcase~1.1.} $i=0$. That is $b-a = 2kn$ for some $k\in\mathbb{Z}$. Then
\begin{align*}
b-a+\lfloor (b-a)/n \rfloor +1 = 1+ 2k(n+1).
\end{align*}
It follows that the east edge of $c$ in $\dg{\eta}$ is labeled by the south step $p_{1+2k(n+1)} = 1$ by the definition of edge sequence. Since
\[
b-a+\lfloor (b-a)/n \rfloor = 2k(n+1),
\]
$p_{2k(n+1)}$ is undefined. Hence, the north edge of $c$ in $\dg{\eta}$ is labeled by the east step $p_{-1+2k(n+1)} = 0$. Meanwhile, in the edge sequence of $\kappa$, we have
\[
p_{1+2k(n+1)} = 0, \quad p_{-1+2k(n+1)} = 1,
\]
which means that adding the box $c$ to $\dg{\kappa}$ corresponds precisely to
interchanging $p_{-1+2k(n+1)}$ and $p_{1+2k(n+1)}$.
Note that
\begin{equation}
s_i(j+2x(n+1)) = s_i(j)+2x(n+1)
\mlabel{eq:period}
\end{equation}
for any $i\in[0,n]$ and $x\in\mathbb{Z}$.
Thus,
\begin{align*}
s_i \big(p_{-1+2x(n+1)}\big)=&\ s_0 \big(p_{-1+2x(n+1)}\big)\hspace{1cm}(\text{by $i=0$})\\
=&\ p_{s_0(-1+2x(n+1))}\\
=&\ p_{s_0(-1)+2x(n+1)}\hspace{1cm}(\text{by~(\ref{eq:period})})\\
=&\ p_{1+2x(n+1)}\hspace{1cm}(\text{by Lemma~\ref{lem:Cperm}~(\ref{it:Cperm2})}),
\end{align*}
for all $x\in\mathbb{Z}$.
Similarly,
$$s_i\big(p_{1+2x(n+1)}\big) = p_{-1+2x(n+1)}$$
and
$$s_i\big(p_{j}\big) = p_j\,\text{ if  }\, j\notin\{ -1+2x(n+1),1+2x(n+1) \mid x\in\mathbb{Z}\}.$$
Hence, the action of $s_i$ on edge sequence only exchanges $p_{-1+2x(n+1)}$ and $p_{1+2x(n+1)}$ for all $x\in\mathbb{Z}$ (especially $x=k$), and further leads to adding $c$ to $\dg{\kappa}$.
Due to the arbitrariness of box $c$, we have $s_i\big( p(\kappa)\big) = p(\eta)$, which follows that
\[
s_i*\kappa = p^{-1}\Big(s_i\big( p(\kappa)\big)\Big) = \eta.
\]

\noindent{\bf Subcase~1.2.} $i=n$. The proof of $i=n$ is similar to the proof of Subcase~1.1, since $s_0$ and $s_n$ are symmetric.

\noindent{\bf Subcase~1.3.} $i\in[1,n-1]$. Suppose $b-a = i +2kn$ for some $k\in\mathbb{Z}$.
Then
\begin{align*}
b-a + \lfloor (b-a)/n \rfloor +1 =&\ i + 2kn + \lfloor (i+2kn)/n \rfloor +1\\
=&\ i + 2kn + 2k +1 \hspace{1cm} (\text{by $i\in[1,n-1]$})\\
=&\ i+1 +2k(n+1),\\
b-a + \lfloor (b-a)/n \rfloor =&\ i + 2kn + \lfloor (i+2kn)/n \rfloor \\
=&\ i +2k(n+1) \nequiv 0\,\text{ mod }\,(n+1).
\end{align*}
Thus, the east (resp. north) edge of $c$ in $\dg{\eta}$ is labeled by the south (resp. east) step
$$
p_{i+1+2k(n+1)} = 1 \,\text{(resp. } p_{i+2k(n+1)} = 0).
$$
On the other hand, $$p_{i+1+2k(n+1)} = 0,  \qquad p_{i+2k(n+1)} = 1$$ in $\dg{\kappa}$.
Hence, adding the box $c$ to $\dg{\kappa}$ is equivalent to
interchanging $p_{i+2k(n+1)}$ and $p_{i+1+2k(n+1)}$.
By the symmetry of $\kappa$ in $\Cco$, $c':=(b,a)$ is an addable corner of $\dg{\kappa}$ with residue $2n-i$.

Following the same analysis as for $c$, we find that adding the box $c'$
to $\dg{\kappa}$ corresponds precisely to interchanging
$p_{-(i+1)-2k(n+1)}$ and $p_{-i-2k(n+1)}$.

By Lemma~\ref{lem:Cperm}~(\ref{it:Cperm1}), we have
\begin{equation*}
s_i(p_x) =
\left\{
\begin{array}{ll}
p_x, & \quad \text{if $x \neq \pm i, \pm (i+1)\,{\rm mod }\, (2n+2)$},\\
p_{x+1}, & \quad \text{if $x = i, -(i+1)\,{\rm mod }\, (2n+2)$},\\
p_{x-1}, & \quad \text{if $x = i+1, -i\,{\rm mod }\, (2n+2)$}.
\end{array}
\right.
\end{equation*}
It induces that the action of $s_i$ on edge sequence leads to adding $c$ and $c'$ to $\dg{\kappa}$. By the arbitrariness of box $c$, $s_i*\kappa=\eta$ as required.

\noindent{\bf Case~2.} The $\kappa$ has a removable corner $c:= (a,b)$ of residue $i \equiv b-a$ mod $2n$. The proof is similar to Case~1, since addable and removable are inverse processes.

\noindent{\bf Case~3.} The $\kappa$ has no addable or removable corner $c:= (a,b)$ of residue $i \equiv b-a$ mod $2n$. The statement is true by $s_i * \kappa = \kappa = s_i\cdot\kappa$.
\end{proof}

The following example illustrates Proposition~\ref{prop:twoaction} by showing explicitly how the action of a simple reflection on the edge sequence agrees with its action on the corresponding symmetric core.

\begin{exam}
Let $n=2$ and $\kappa:=(5,4,3,2,1)\in\Cco$. It follows from Example~\ref{ex:bij1} that $$s_1\cdot\kappa = (6,5,4,3,2,1).$$  Together with Example~\ref{ex:edgeseq}, we have
\[
\begin{tabular}{c | c c c c c c c c c c c c c c c c c c c}
i & -8 & -7 & $\widehat{-6}$ & -5 & -4 & $\widehat{-3}$ & -2 & -1 & $\widehat{0}$ & 1 & 2 & $\widehat{3}$ & 4 & 5 & $\widehat{6}$ & 7 & 8 & $\widehat{9}$ & 10\\
\hline
$p_i$ & 1 & 0 & & 1 & 0 & & 1 & 0 & & 1 & 0 & & 1 & 0 & & 1 & 0 & & 0 \\
$s_1(p_i)$ & 0 & 1 & & 0 & 1 & & 0 & 1 & & 0 & 1 & & 0 & 1 & & 0 & 1 & & 0
\end{tabular}
\]
It follows that $$s_1 * \kappa = (6,5,4,3,2,1) = s_1\cdot\kappa$$ via
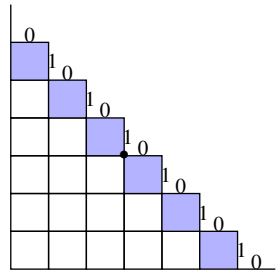
\[
\begin{tikzpicture}[scale=.5,line width=0.5pt,baseline=(a.base)]
\draw[draw=black] (-1,-1) rectangle (0,0);
\draw[draw=black] (0,-1) rectangle (1,0);
\draw[draw=black] (-1,0) rectangle (0,1);
\draw[draw=black] (1,-1) rectangle (2,0);
\draw[draw=black] (-1,1) rectangle (0,2);
\draw[draw=black] (0,0) rectangle (1,1);
\draw[draw=black] (-1,2) rectangle (0,3);
\draw[draw=black] (2,-1) rectangle (3,0);
\draw[draw=black] (0,1) rectangle (1,2);
\draw[draw=black] (1,0) rectangle (2,1);
\draw[draw=black] (3,-1) rectangle (4,0);
\draw[draw=black] (-1,3) rectangle (0,4);
\draw[draw=black] (1,1) rectangle (2,2);
\draw[draw=black] (0,2) rectangle (1,3);
\draw[draw=black] (2,0) rectangle (3,1);
\filldraw[blue!30,draw=black] (-1,4) rectangle (0,5);
\filldraw[blue!30,draw=black] (0,3) rectangle (1,4);
\filldraw[blue!30,draw=black] (1,2) rectangle (2,3);
\filldraw[blue!30,draw=black] (2,1) rectangle (3,2);
\filldraw[blue!30,draw=black] (3,0) rectangle (4,1);
\filldraw[blue!30,draw=black] (4,-1) rectangle (5,0);
\node at(-0.5,5.2){\tiny\( 0 \)}; \node at(0.1,4.5){\tiny\( 1 \)};
\node at(0.5,4.2){\tiny\( 0 \)}; \node at(1.1,3.5){\tiny\( 1 \)};
\node at(1.5,3.2){\tiny\( 0 \)}; \node at(2.1,2.5){\tiny\( 1 \)};
\node at(2.5,2.2){\tiny\( 0 \)}; \node at(3.1,1.5){\tiny\( 1 \)};
\node at(3.5,1.2){\tiny\( 0 \)}; \node at(4.1,0.5){\tiny\( 1 \)};
\node at(4.5,0.2){\tiny\( 0 \)}; \node at(5.1,-0.5){\tiny\( 1 \)};
\node at(5.5,-0.8){\tiny\( 0 \)};
\node (a) [align=center] {\\[0pt] };
\node at(2,2){\tiny\( \bullet \)};
\draw[black,line width=0.5pt] (5,-1)--(6,-1);
\draw[black,line width=0.5pt] (-1,5)--(-1,6);
\end{tikzpicture}
\]
\end{exam}

\subsection{Type C transposition}
In this subsection, following the permutation-theoretic approach of~\cite{LLMSSZ,LLMS}, we introduce type C transpositions in $\widetilde{C}_n$ and use them to develop a permutation model for type C strong marked tableaux.

\begin{defn}
For $i,j\in\mathbb{Z}\setminus\{ q(n+1),q\in\mathbb{Z}\}$ with $i\nequiv j$ mod $(2n+2)$, define the {\em type C transposition} $\trans{i}{j}$ to be the unique element of $\tilc$ such that
\begin{equation*}
\trans{i}{j}(x) =
\left\{
\begin{array}{ll}
j+q(2n+2), & \quad \text{if $x = i+q(2n+2)$ for $q\in\mathbb{Z}$},\\
i+q(2n+2), & \quad \text{if $x = j+q(2n+2)$ for $q\in\mathbb{Z}$},\\
-j+q(2n+2), & \quad \text{if $x = -i+q(2n+2)$ for $q\in\mathbb{Z}$},\\
-i+q(2n+2), & \quad \text{if $x = -j+q(2n+2)$ for $q\in\mathbb{Z}$},\\
x, & \quad \text{otherwise}.
\end{array}
\right.
\end{equation*}
\mlabel{defn:trans}
\end{defn}

The following remark collects several elementary identities that will be useful in the permutation calculations below.

\begin{remark}The type C transpositions introduced above satisfy the following basic relations:
\begin{enumerate}
\item For $i,j\in\mathbb{Z}\setminus\{ q(n+1),q\in\mathbb{Z}\}$ with $i\nequiv j$ mod $(2n+2)$,
    \begin{equation}
    \trans{i}{j} = \trans{j}{i} = \trans{-i}{-j} =\trans{-j}{-i}.
    \mlabel{eq:tfact1}
    \end{equation}

\item For each $q\in\mathbb{Z}$,
\begin{equation}
\trans{i}{j} = \trans{i+q(2n+2)}{j+q(2n+2)}.
\mlabel{eq:tfact2}
\end{equation}

\item By Lemma~\ref{lem:Cperm}, for $i\in[0,n]$,
\begin{equation}
s_i =
\left\{
\begin{array}{ll}
\trans{i}{i+1}, & \quad \text{if $i\in[1,n-1]$},\\
\trans{-1}{1}, & \quad \text{if $i=0$},\\
\trans{n}{n+2}, & \quad \text{if $i=n$}.
\end{array}
\right.
\mlabel{eq:tfact3}
\end{equation}
\end{enumerate}
\mlabel{re:tfact}
\end{remark}

We define some expressions in $\tilc$ inspired by~\cite[\S6.1]{LSS10c}:
\[
\begin{array}{ll}
I\uparrow^k_{k'}:= s_{k'}s_{k'+1}\cdots s_{k-1}s_{k}, & \quad \text{for $0\leq k'\leq k\leq n$},\\
I\downarrow^k_{k'}:= s_{k}s_{k-1}\cdots s_{k'+1}s_{k'}, & \quad \text{for $0\leq k'\leq k\leq n$},\\
V^{k,k'}_{a}:= s_k s_{k-1}\cdots s_{a+1} s_{a} s_{a+1} \cdots s_{k'-1} s_{k'}, & \quad \text{for $0\leq a \leq k, k' \leq n$}, \\
\Lambda_{k,k'}^{a}:= s_k s_{k+1}\cdots s_{a-1} s_{a} s_{a-1} \cdots s_{k'+1} s_{k'}, & \quad \text{for $0\leq  k, k' \leq a \leq n$}, \\
\reW{a}{a'}{k}{k'}{k''}:=V^{k,k'-1}_{a}s_{k'}V^{k'-1,k''}_{a'},  & \quad\text{for $0\leq  a,a' \leq k,k',k'' \leq n$},\\
\reM{a}{a'}{k}{k'}{k''}:=\Lambda_{k,k'+1}^{a}s_{k'}\Lambda_{k'+1,k''}^{a'},  & \quad\text{for $0 \leq k,k',k'' \leq a,a' \leq n$}.
\end{array}
\]
The name $I\uparrow$ (similar to $I\downarrow,V,\Lambda,W,M$) is defined by that, the generator indices in the word go up like the letter $I$.

\begin{remark}
With the notation introduced above, the type C transpositions can be expressed in terms of the simple reflections $s_0,s_1,\ldots,s_n$ as follows, according to the relative positions of their indices.

\noindent{\bf Case~A.} $i,j \in[-n,n]\setminus\{0\}$. By~(\ref{eq:tfact1}), we only need to consider the case of $|i|\leq j$.

\noindent{\bf Subcase~A.1.} $1\leq i<j\leq n$. Then
\begin{equation}
\trans{i}{j} = V^{j-1,j-1}_{i}.
\mlabel{eq:t11}
\end{equation}

\noindent{\bf Subcase~A.2.} $-n\leq i <0 <j \leq n$ and $|i|=j$. Then
\begin{equation}
\trans{i}{j} = V^{j-1,j-1}_{0}.
\mlabel{eq:t12}
\end{equation}

\noindent{\bf Subcase~A.3.} $-n\leq i <0 <j \leq n$ and $|i|<j$. Then
\begin{equation}
\trans{i}{j} = \reW{0}{0}{j-1}{-i}{j-1}.
\mlabel{eq:t13}
\end{equation}

\noindent{\bf Case~B.} $i,j \in[1,2n+1]\setminus\{n+1\}$. By~(\ref{eq:CpermDesSym}) and~(\ref{eq:tfact1}), we only need to consider the case of $|i-(n+1)|\leq j-(n+1)$.

\noindent{\bf Subcase~B.1.} $n+2\leq i<j\leq 2n+1$. Then
\begin{equation}
\trans{i}{j} = \Lambda_{2n+2-j,2n+2-j}^{2n+2-i-1}.
\mlabel{eq:t21}
\end{equation}

\noindent{\bf Subcase~B.2.} $1\leq i <n+1 <j \leq 2n+1$ and $|i-(n+1)|=j-(n+1)$. Then
\begin{equation}
\trans{i}{j} = \Lambda^{n}_{2n+2-j,2n+2-j}.
\mlabel{eq:t22}
\end{equation}

\noindent{\bf Subcase~B.3.} $1\leq i <n+1 <j \leq 2n+1$ and $|i-(n+1)|<j-(n+1)$. Then
\begin{equation}
\trans{i}{j} = \reM{n}{n}{2n+2-j}{i-1}{2n+2-j}.
\mlabel{eq:t23}
\end{equation}

\noindent{\bf Case~C.} Otherwise, that is
$$j\in[1,n], \qquad i\in\mathbb{Z}_{<0}\setminus(\{ q(n+1),q\in\mathbb{Z}_{<0}\}\cup [-n,1] ),$$
or
$$j\in[n+2,2n+1],\qquad i\in\mathbb{Z}_{<0}\setminus\{ q(n+1),q\in\mathbb{Z}_{<0}\}. $$
By~\eqref{eq:tfact1} and~\eqref{eq:tfact2}, we can always assume that
$$i<j, \qquad j\in[1,2n+1]\setminus (n+1).$$
Then $\trans{i}{j}$ can be expressed by~(\ref{eq:t11})-(\ref{eq:t23}) via:
fix $i,j$, there exist an $$a\in[1,2n+1]\setminus\{n+1,j\}$$ such that $$a= i+2q(n+1)\,\text{ for some }\, q\in\mathbb{Z}_{>0}.$$
Then
\begin{equation}
\trans{i}{j} =
\Big(\prod_{x=0}^{q-1}\trans{i+2x(n+1)}{\bar{x}}
\trans{\bar{x}}{i+2(x+1)(n+1)}\Big)\trans{a}{j}
\Big(\prod_{x=q-1}^{0}
\trans{\bar{x}}{i+2(x+1)(n+1)}\trans{i+2x(n+1)}{\bar{x}}\Big),
\mlabel{eq:t3}
\end{equation}
where
$$
\bar{x}:=
\left\{
\begin{array}{ll}
-i+(-4q+2+2x)(n+1), & \quad \text{if $a\in[1,n]$},\\
-i+(-4q+4+2x)(n+1), & \quad \text{if $a\in[n+2,2n+1]$}.\\
\end{array}
\right.
$$
Note that each type C transposition in the right hand side of~\eqref{eq:t3} can be expressed as one of~(\ref{eq:t11})-(\ref{eq:t23}) due to~\eqref{eq:tfact2}.
\mlabel{rk:3cases}
\end{remark}

The following two examples illustrate the preceding descriptions: the first gives an explicit expression of a type C transposition in terms of simple reflections, while the second shows how a transposition in Case~C can be reduced to those covered by Cases~A and~B.

\begin{exam}
Let $n=4$. Consider the action of
$$\reW{0}{0}{3}{3}{3} = s_3s_2s_1s_0s_1s_2s_3s_2s_1s_0s_1s_2s_3$$
on $\mathbb{Z}$:
\[
\begin{tabular}{c | c c c c c c c c c c c}
$\mathbb{Z}$ & -5 & -4 & -3 & -2 & -1 & 0 & 1 & 2 & 3 & 4 & 5 \\
\hline
& -5 & \textcolor{red}{3} & \textcolor{red}{4} & -2 & -1 & 0 & 1 & 2 & \textcolor{red}{-4} & \textcolor{red}{-3} & 5
\end{tabular}
\]
Then  $\trans{-3}{4}=\reW{0}{0}{3}{3}{3}$.
\end{exam}

\begin{exam}
Let $j\in[1,n]$ and $$i\in[-2(n+1)+1, -(n+1)-1]\,\text{ such that }\, j\nequiv i \mod (2n+2).$$ Suppose $$a:=i+2(n+1).$$
Then
$$a\in[1,n], \qquad \trans{i}{j} = \trans{i}{-i-2(n+1)}\trans{-i-2(n+1)}{i+2(n+1)}\trans{a}{j}
\trans{-i-2(n+1)}{i+2(n+1)}\trans{i}{-i-2(n+1)},$$
where
\[
\begin{tabular}{c | c }
$\trans{i}{-i-2(n+1)} = \trans{i+2(n+1)}{-i}$ &  $i+2(n+1)=a\in[1,n]$ and $-i\in[n+2,2n+1]$  \\
\hline
$\trans{-i-2(n+1)}{i+2(n+1)}$ & $-i-2(n+1)\in[-n,-1]$ and $i+2(n+1)=a\in[1,n]$ \\
\hline
$\trans{a}{j}$ & $a\in[1,n]$, $j\in[1,n]$ and $a\neq j$
\end{tabular}
\]
\end{exam}

\begin{remark}
Note that Cases~A and~B in Remark~\mref{rk:3cases} are symmetric, as the same way as the symmetry of $s_0$ and $s_n$. Besides, Case~C can be represented by Cases~A and~B. Hence, for the rest of this paper on $\trans{i}{j}$, we only consider Case~A.
\mlabel{re:3cases1}
\end{remark}

The preceding descriptions express type C transpositions in terms of the
simple reflections. The following theorem shows that these transpositions
are compatible with the action of the affine Weyl group through conjugation,
a property that will be fundamental for the permutation model developed below.

\begin{theorem}
For $w\in \tilc$ and $i,j\in\mathbb{Z}\setminus\{ q(n+1),q\in\mathbb{Z}\}$ with $i\nequiv j \mod (2n+2)$, we have
\begin{equation}
w\trans{i}{j} = \trans{w(i)}{w(j)}w.
\mlabel{eq:reflection}
\end{equation}
\mlabel{thm:reflection}
\end{theorem}

\begin{proof}
By Remark~\mref{re:3cases1}, it suffices to consider Case~A in Remark~\mref{rk:3cases}, that is, $i,j \in[-n,n]\setminus\{0\}$ with $|i|\leq j$. So we only need to prove
\begin{equation}
s_a \trans{i}{j} = \trans{s_a(i)}{s_a(j)} s_a\,\text{ for each }\, a\in[0,n].
\mlabel{eq:refs}
\end{equation}
The proof mainly employs the generating relations~(\ref{eq:Cgene}) of $s_0,\ldots,s_n$.
Under the assumptions $i,j\in[-n,n]\setminus\{0\}$ and $|i|\leq j$, we
necessarily have $j>0$. There are therefore exactly three mutually exclusive
possibilities: either $0<i<j$, or $i<0<j$ with $|i|=j$, or $i<0<j$ with
$|i|<j$. Accordingly, we divide the proof into the following three cases.

\noindent{\bf Case~1.} $1\leq i<j\leq n$. We divided the proof into the following subcases.

\noindent{\bf Subcase~1.1.} $0=a<i=a+1 =1$. Then $s_a(i) = s_0(1) = -1$, $s_a(j) =s_0(j) = j$  and
\begin{align*}
s_a \trans{i}{j} = &\ s_0 V^{j-1,j-1}_{i} \hspace{1cm}(\text{by~(\ref{eq:t11})})\\
=&\ s_0 I\downarrow^{j-1}_{2} s_1 I\uparrow^{j-1}_{2} \\
=&\ I\downarrow^{j-1}_{2} s_0 s_1 I\uparrow^{j-1}_{2}\hspace{1cm}(\text{$s_0 s_b = s_b s_0$ for $b\in[2,n]$})\\
=&\ I\downarrow^{j-1}_{2} (s_1s_0s_0s_1)s_0 s_1 I\uparrow^{j-1}_{2}\hspace{1cm}(\text{$s_a^2=1$})\\
=&\ I\downarrow^{j-1}_{1} s_0 (s_0 s_1 s_0 s_1)I\uparrow^{j-1}_{2}\\
=&\ I\downarrow^{j-1}_{1} s_0 (s_1 s_0 s_1 s_0)I\uparrow^{j-1}_{2} \hspace{1cm}(\text{$s_0 s_1 s_0 s_1 = s_1 s_0 s_1 s_0$})\\
=&\ I\downarrow^{j-1}_{1} s_0 s_1 s_0 s_1 I\uparrow^{j-1}_{2} s_0 \hspace{1cm}(\text{$s_0 s_b = s_b s_0$ for $b\in[2,n]$})\\
=&\ \reW{0}{0}{j-1}{1}{j-1} s_0\\
=&\ \trans{-1}{j}s_a \hspace{1cm}(\text{by~(\ref{eq:t13})})\\
=&\ \trans{s_a(i)}{s_a(j)}s_a.
\end{align*}

\noindent{\bf Subcase~1.2.} $0<a<i=a+1$. Then $s_a(i) = s_a(a+1) = a$, $s_a(j) = j$ and \begin{align*}
s_a \trans{i}{j} = &\ s_a V^{j-1,j-1}_{a+1} \hspace{1cm}(\text{by~(\ref{eq:t11})})\\
=&\ s_a I\downarrow^{j-1}_{a+2} s_{a+1} I\uparrow^{j-1}_{a+2}\\
=&\ I\downarrow^{j-1}_{a+2} s_a s_{a+1} I\uparrow^{j-1}_{a+2}
\hspace{1cm}(\text{$s_a s_b = s_b s_a$ for $b\in[a+2,n]$})\\
=&\ I\downarrow^{j-1}_{a+2} s_a s_{a+1} (s_a s_a) I\uparrow^{j-1}_{a+2}\hspace{1cm}(\text{$s_a^2=1$})\\
=&\ I\downarrow^{j-1}_{a+2} (s_{a+1} s_a s_{a+1})  s_a I\uparrow^{j-1}_{a+2}
\hspace{1cm}(\text{$s_a s_{a+1} s_a=s_{a+1} s_a s_{a+1}$ for $a\in[1,n-1]$})\\
=&\ I\downarrow^{j-1}_{a+2} s_{a+1} s_a s_{a+1} I\uparrow^{j-1}_{a+2} s_a
\hspace{1cm}(\text{$s_a s_b = s_b s_a$ for $b\in[a+2,n]$})\\
=&\ V_{a}^{j-1,j-1}s_a\\
=&\ \trans{a}{j}s_a \hspace{1cm}(\text{by~(\ref{eq:t11})})\\
=&\ \trans{s_a(i)}{s_a(j)}s_a.
\end{align*}

\noindent{\bf Subcase~1.3.} $0<a<a+1<i$. Then $s_a(i)=i$, $s_a(j)=j$ and
\begin{align*}
s_a \trans{i}{j} \overset{(\ref{eq:t11})}{=} s_a V^{j-1,j-1}_{i}
\overset{(\ref{eq:Cgene})}{=} V^{j-1,j-1}_{i} s_a \overset{(\ref{eq:t11})}{=}\trans{s_a(i)}{s_a(j)} s_a.
\end{align*}

\noindent{\bf Subcase~1.4.} $0<a=i$. Then $s_a(i) = i+1$, $s_a(j) = j$ and
\begin{align*}
s_a \trans{i}{j} = &\ s_i V^{j-1,j-1}_{i} \hspace{1cm}(\text{by~(\ref{eq:t11})})\\
=&\ s_i I\downarrow^{j-1}_{i+2} s_{i+1} s_i s_{i+1} I\uparrow^{j-1}_{i+2}\\
=&\ I\downarrow^{j-1}_{i+2} s_i s_{i+1} s_i s_{i+1} I\uparrow^{j-1}_{i+2}
\hspace{1cm}(\text{$s_i s_b = s_b s_i$ for $b\in[i+2,n]$})\\
=&\ I\downarrow^{j-1}_{i+2} s_{i+1} s_i I\uparrow^{j-1}_{i+2}
\hspace{1cm}(\text{$s_{i+1}^2=1$})\\
=&\ I\downarrow^{j-1}_{i+2} s_{i+1} I\uparrow^{j-1}_{i+2} s_i
\hspace{1cm}(\text{$s_i s_b = s_b s_i$ for $b\in[i+2,n]$})\\
=&\ V^{j-1,j-1}_{i+1}s_i\\
=&\ \trans{i+1}{j} s_i\hspace{1cm}(\text{by~(\ref{eq:t11})})\\
=&\ \trans{s_a(i)}{s_a(j)} s_a.
\end{align*}

\noindent{\bf Subcase~1.5.} $i<a<j-1$. Then $s_a(i) =i$, $s_a(j)=j$ and
\begin{align*}
s_a \trans{i}{j} = &\ s_a V^{j-1,j-1}_{i} \hspace{1cm}(\text{by~(\ref{eq:t11})})\\
=&\ s_a I\downarrow^{j-1}_{a+2} s_{a+1} s_a V_{i}^{a-1,a-1} s_a s_{a+1} I\uparrow^{j-1}_{a+2}\\
=&\ I\downarrow^{j-1}_{a+2} s_a s_{a+1} s_a V_{i}^{a-1,a-1} s_a s_{a+1} I\uparrow^{j-1}_{a+2}
\hspace{1cm}(\text{$s_a s_b = s_b s_a$ for $b\in[a+2,n]$})\\
=&\ I\downarrow^{j-1}_{a+2} s_{a+1} s_a s_{a+1} V_{i}^{a-1,a-1} s_a s_{a+1} I\uparrow^{j-1}_{a+2}
\hspace{1cm}(\text{$s_a s_{a+1} s_a =s_{a+1} s_a s_{a+1}$})\\
=&\ I\downarrow^{j-1}_{a+2} s_{a+1} s_a  V_{i}^{a-1,a-1} s_{a+1} s_a s_{a+1} I\uparrow^{j-1}_{a+2}
\hspace{1cm}(\text{$s_{a+1} s_b = s_b s_{a+1}$ for $b\in[0,a-1]$})\\
=&\ I\downarrow^{j-1}_{a+2} s_{a+1} s_a  V_{i}^{a-1,a-1} s_a s_{a+1} s_a I\uparrow^{j-1}_{a+2}
\hspace{1cm}(\text{$s_{a+1} s_a s_{a+1} = s_a s_{a+1} s_a$})\\
=&\ I\downarrow^{j-1}_{a+2} s_{a+1} s_a  V_{i}^{a-1,a-1} s_a s_{a+1} I\uparrow^{j-1}_{a+2} s_a
\hspace{1cm}(\text{$s_a s_b = s_b s_a$ for $b\in[a+2,n]$})\\
=&\ V^{j-1,j-1}_{i} s_a\\
=&\ \trans{s_a(i)}{s_a(j)} s_a
\hspace{1cm}(\text{by~(\ref{eq:t11})}).
\end{align*}

\noindent{\bf Subcase~1.6.} $a=j-1$. Then $s_a(i) = i$, $s_a(j) = j-1$ and
\begin{align*}
s_a \trans{i}{j} = &\ s_{j-1} V^{j-1,j-1}_{i} \hspace{1cm}(\text{by~(\ref{eq:t11})})\\
=&\ s_{j-1} s_{j-1} V^{j-2,j-2}_{i} s_{j-1}\\
=&\ V^{j-2,j-2}_{i} s_{j-1}
\hspace{1cm}(\text{$s_{j-1}^2=1$})\\
=&\ \trans{i}{j-1}s_{j-1}\hspace{1cm}(\text{by~(\ref{eq:t11})})\\
=&\ \trans{s_a(i)}{s_a(j)} s_a.
\end{align*}

\noindent{\bf Subcase~1.7.} $a=j<n$. Then $s_a(i) = i$, $s_a(j) = j+1$ and
\begin{align*}
s_a \trans{i}{j} =&\ s_j V^{j-1,j-1}_{i} \hspace{1cm}(\text{by~(\ref{eq:t11})})\\
=&\ s_j V^{j-1,j-1}_{i} s_j s_j \hspace{1cm}(\text{$s_{j-1}^2=1$})\\
=&\ V^{j,j}_{i} s_j\\
=&\ \trans{s_a(i)}{s_a(j)} s_a \hspace{1cm}(\text{by~(\ref{eq:t11})}).
\end{align*}

\noindent{\bf Subcase~1.8.} $a=j=n$. Then $s_a(i) = i$ and $s_a(j) = s_n(n) = n+2$.
Note that
\begin{equation}
\begin{aligned}
\trans{s_a(i)}{s_a(j)} = &\ \trans{i}{n+2}\\
=&\ \trans{-(n+2) + 2n+2}{-i+2n+2}\hspace{1cm}(\text{by~(\ref{eq:tfact1}) and~(\ref{eq:tfact2})})\\
=&\ \trans{n}{-i+2n+2} \\
=&\ \reM{n}{n}{2n+2-(-i+2n+2)}{n-1}{2n+2-(-i+2n+2)}
\hspace{1cm}(\text{by~(\ref{eq:t23})})\\
=&\ \reM{n}{n}{i}{n-1}{i}.
\end{aligned}
\mlabel{eq:a=j=n}
\end{equation}
For each $i,j\in[1,n]$ such that $i<j$, we have
\begin{align*}
V^{j-1,j-1}_{i} =&\ I\downarrow^{j-1}_{i+2} s_{i+1} s_i s_{i+1} I\uparrow^{j-1}_{i+2}\\
=&\ I\downarrow^{j-1}_{i+2} s_i s_{i+1} s_i I\uparrow^{j-1}_{i+2}
\hspace{1cm} (\text{$s_{i+1} s_i s_{i+1} = s_i s_{i+1} s_i$})\\
=&\ s_i I\downarrow^{j-1}_{i+2} s_{i+1} I\uparrow^{j-1}_{i+2} s_i
\hspace{1cm} (\text{$s_i s_b = s_b s_i$ for $b\in [i+2,n]$})\\
=&\ s_i V^{j-1,j-1}_{i+1} s_i,
\end{align*}
which follows that
\begin{equation}
V^{j-1,j-1}_{i} = \Lambda^{j-1}_{i,i}\,\text{ for }\, i,j\in[1,n],\, i<j.
\mlabel{eq:V=Lam}
\end{equation}
Thus,
\begin{align*}
s_a\trans{i}{j} =&\ s_n V^{n-1,n-1}_{i} \hspace{1cm}(\text{by~(\ref{eq:t11})})\\
=&\ s_n \Lambda^{n-1}_{i,i}\hspace{1cm}(\text{by~(\ref{eq:V=Lam})})\\
=&\ s_n I\uparrow_{i}^{n-2} s_{n-1} I\downarrow_{i}^{n-2}\\
=&\ I\uparrow_{i}^{n-2} s_n s_{n-1} I\downarrow_{i}^{n-2}
\hspace{1cm} (\text{$s_n s_b = s_b s_n$ for $b\in [0,n-2]$})\\
=&\ I\uparrow_{i}^{n-2} s_n s_{n-1} (s_n s_{n-1} s_{n-1} s_n) I\downarrow_{i}^{n-2}
\hspace{1cm} (\text{$s_i^2 = 1$})\\
=&\ I\uparrow_{i}^{n-2} s_n s_{n-1} s_n s_{n-1} s_{n-1}  I\downarrow_{i}^{n-2} s_n
\hspace{1cm} (\text{$s_n s_b = s_b s_n$ for $b\in [0,n-2]$})\\
=&\ I\uparrow_{i}^{n-2}  s_{n-1} s_n s_{n-1} s_n s_{n-1}  I\downarrow_{i}^{n-2} s_n
\hspace{1cm} (\text{$s_n s_{n-1} s_n s_{n-1} = s_{n-1} s_n s_{n-1} s_n$})\\
=&\ \reM{n}{n}{i}{n-1}{i} s_n\\
=&\ \trans{s_a(i)}{s_a(j)}s_a \hspace{1cm}(\text{by~(\ref{eq:a=j=n})}).
\end{align*}

\noindent{\bf Subcase~1.9.} $a>j$. Then $s_a(i) = i$, $s_a(j) = j$ and
\begin{align*}
s_a\trans{i}{j} \overset{(\ref{eq:t11})}{=} s_a V^{j-1,j-1}_{i} \overset{(\ref{eq:Cgene})}{=} V^{j-1,j-1}_{i} s_a \overset{(\ref{eq:t11})}{=} \trans{i}{j}s_a = \trans{s_a(i)}{s_a(j)} s_a.
\end{align*}

\noindent{\bf Case~2.} $-n\leq i <0 <j \leq n$ and $|i|=j$. The proof in this case is divided into the following subcases.

\noindent{\bf Subcase~2.1.} $0=a<j=a+1=1$. Then $s_a(i) = s_0(-1)=1$, $s_a(j)= s_0(1)=-1$ and
\begin{align*}
s_a\trans{i}{j} \overset{(\ref{eq:t12})}{=} s_0 V_{0}^{j-1,j-1} = s_0 s_0 = V_{0}^{j-1,j-1} s_0 = \trans{i}{j}s_a = \trans{s_a(j)}{s_a(i)}s_a \overset{(\ref{eq:tfact1})}{=} \trans{s_a(i)}{s_a(j)}s_a.
\end{align*}

\noindent{\bf Subcase~2.2.} $0 = a <1<j$. Then $s_a(i) = i$, $s_a(j) = j$ and
\begin{align*}
s_a\trans{i}{j} = &\ s_0 V^{j-1,j-1}_0 \hspace{1cm}(\text{by~(\ref{eq:t12})})\\
=&\ s_0 I\downarrow^{j-1}_2 s_1 s_0 s_1 I\uparrow^{j-1}_2 \\
=&\ I\downarrow^{j-1}_2 s_0 s_1 s_0 s_1 I\uparrow^{j-1}_2 \hspace{1cm} (\text{$s_0 s_b = s_b s_0$ for $b\in [2,n]$})\\
=&\ I\downarrow^{j-1}_2 s_1 s_0 s_1 s_0 I\uparrow^{j-1}_2 \hspace{1cm} (\text{$s_0 s_1 s_0 s_1 = s_1 s_0 s_1 s_0$})\\
=&\ I\downarrow^{j-1}_2 s_1 s_0 s_1 I\uparrow^{j-1}_2 s_0 \hspace{1cm} (\text{$s_0 s_b = s_b s_0$ for $b\in [2,n]$})\\
=&\ V^{j-1,j-1}_0 s_0\\
=&\ \trans{s_a(i)}{s_a(j)} s_a.
\end{align*}

\noindent{\bf Subcase~2.3.} $0< a <j-1$. Then $s_a(i) =i$, $s_a(j)=j$ and
\begin{align*}
s_a\trans{i}{j} = &\ s_a V^{j-1,j-1}_0 \hspace{1cm}(\text{by~(\ref{eq:t12})})\\
=&\ s_a I\downarrow^{j-1}_{a+2} s_{a+1} s_a V_0^{a-1,j-1}\\
=&\ I\downarrow^{j-1}_{a+2} s_a s_{a+1} s_a V_0^{a-1,j-1} \hspace{1cm} (\text{$s_a s_b = s_b s_a$ for $b\in [a+2,n]$})\\
=&\ I\downarrow^{j-1}_{a+2} s_{a+1} s_a s_{a+1} V_0^{a-1,j-1} \hspace{1cm} (\text{$s_a s_{a+1} s_a = s_{a+1} s_a s_{a+1} $})\\
=&\ I\downarrow^{j-1}_{a} s_{a+1} V_0^{a-1,a-1} s_a s_{a+1} I\uparrow^{j-1}_{a+2}\\
=&\ I\downarrow^{j-1}_{a} V_0^{a-1,a-1} s_{a+1} s_a s_{a+1} I\uparrow^{j-1}_{a+2} \hspace{1cm} (\text{$s_{a+1} s_b = s_b s_{a+1}$ for $b\in [0,a-1]$})\\
=&\ I\downarrow^{j-1}_{a} V_0^{a-1,a-1} s_a s_{a+1} s_a  I\uparrow^{j-1}_{a+2}
\hspace{1cm} (\text{$ s_{a+1} s_a s_{a+1} = s_a s_{a+1} s_a $})\\
=&\ I\downarrow^{j-1}_{a} V_0^{a-1,a-1} s_a s_{a+1} I\uparrow^{j-1}_{a+2}s_a
\hspace{1cm} (\text{$s_a s_b = s_b s_a$ for $b\in [a+2,n]$})\\
=&\ V^{j-1,j-1}_0 s_a \\
=&\ \trans{s_a(i)}{s_a(j)} s_a.
\end{align*}

\noindent
{\bf Subcase~2.4.} $0<a =j-1$. Then $s_a(i) = s_{j-1}(-j) = -(j-1)$, $s_a(j) = s_{j-1}(j) = j-1$ and
\begin{align*}
s_a\trans{i}{j} =&\ s_{j-1} V^{j-1,j-1}_0 \hspace{1cm}(\text{by~(\ref{eq:t12})})\\
=&\ s_{j-1}s_{j-1}V^{j-2,j-2}_0 s_{j-1}\\
=&\ V^{j-2,j-2}_0 s_{j-1} \hspace{1cm}(\text{$s_{j-1}^2 = 1$})\\
=&\ \trans{s_a(i)}{s_a(j)}s_a \hspace{1cm}(\text{by~(\ref{eq:t12})}).
\end{align*}

\noindent{\bf Subcase~2.5.} $0<a=j<n$. Then $s_a(i) = s_j(-j) = -(j+1)$, $s_a(j) = s_j(j) = j+1$ and
\begin{align*}
s_a\trans{i}{j} \overset{(\ref{eq:t12})}{=} s_j V_{0}^{j-1,j-1} = s_j V_{0}^{j-1,j-1} s_j s_j = V_{0}^{j,j}s_j  \overset{(\ref{eq:t12})}{=} \trans{s_a(i)}{s_a(j)}s_a.
\end{align*}

\noindent{\bf Subcase~2.6.} $a=j=n$. Then $s_a(i) = s_n(-n) = -(n+2)$, $s_a(j) = s_n(n) = n+2$. Note that
\begin{equation}
\trans{s_a(i)}{s_a(j)} = \trans{-(n+2)}{n+2} = \trans{-(n+2)}{-n}\trans{-n}{n}\trans{n}{n+2} = s_n\trans{-n}{n}s_n.
\mlabel{eq:tnn2}
\end{equation}
The last equation is from Remark~\ref{re:tfact}. Thus
$$ \trans{-(n+2)}{-n} = \trans{-(n+2)+2n+2}{-n+2n+2} = \trans{n}{n+2} = s_n, $$
and so
\begin{align*}
s_a\trans{i}{j} =&\ s_{n} V^{n-1,n-1}_0 \hspace{1cm}(\text{by~(\ref{eq:t12})})\\
=&\ s_{n} V^{n-1,n-1}_0 s_{n} s_n \hspace{1cm}(\text{$s_n^2 = 1$})\\
=&\ s_{n} \trans{-n}{-n}s_n s_n \hspace{1cm}(\text{by~(\ref{eq:t12})})\\
=&\ \trans{s_a(i)}{s_a(j)} s_a\hspace{1cm}(\text{by~(\ref{eq:tnn2})}).
\end{align*}

\noindent
{\bf Subcase~2.7.} $a>j$. Then $s_a(i)=i$, $s_a(j)=j$ and
\begin{align*}
s_a\trans{i}{j} \overset{(\ref{eq:t12})}{=} s_a V_{0}^{j-1,j-1} = V_{0}^{j-1,j-1}s_a \overset{(\ref{eq:t12})}{=} \trans{s_a(i)}{s_a(j)} s_a.
\end{align*}
The second equation holds by $s_a s_b = s_b s_a$ for $b\in [0,a-2]$.

\noindent{\bf Case~3.} $-n\leq i <0 <j \leq n$ and $|i|<j$. We show that the statement is true via the subcases as follows.

\noindent{\bf Subcase~3.1.} $ 0 = a < |i| = a + 1 = 1$. Then $s_a(i) = s_0(-1) = 1$, $s_a(j)=j$ and
\begin{align*}
s_a\trans{i}{j} =&\ s_{0} \reW{0}{0}{j-1}{1}{j-1} \hspace{1cm}(\text{by~(\ref{eq:t13})})\\
=&\ s_0 I\downarrow^{j-1}_2 s_1 s_0 s_1 I\uparrow^{j-1}_0\\
=&\ I\downarrow^{j-1}_2 s_0 s_1 s_0 s_1 I\uparrow^{j-1}_0
\hspace{1cm} (\text{$s_0 s_b = s_b s_0$ for $b\in [2,n]$})\\
=&\ I\downarrow^{j-1}_2 s_1 s_0 s_1 s_0 I\uparrow^{j-1}_0
\hspace{1cm} (\text{$s_0 s_1 s_0 s_1 = s_1 s_0 s_1 s_0$})\\
=&\ I\downarrow^{j-1}_1 s_0 (s_1 s_0 s_0 s_1) I\uparrow^{j-1}_2\\
=&\ I\downarrow^{j-1}_1 s_0 I\uparrow^{j-1}_2
\hspace{1cm} (\text{$s_i^2=1$})\\
=&\ I\downarrow^{j-1}_1 I\uparrow^{j-1}_2 s_0
\hspace{1cm} (\text{$s_0 s_b = s_b s_0$ for $b\in [2,n]$})\\
=&\ V_{1}^{j-1,j-1} s_0\\
=&\ \trans{1}{j} s_a \hspace{1cm}(\text{by~(\ref{eq:t11})})\\
=&\ \trans{s_a(i)}{s_a(j)} s_a.
\end{align*}

\noindent{\bf Subcase~3.2.} $ 0 = a < 1 < |i|$. Then $s_a(i) = i$, $s_a(j) = j$ and
\begin{align*}
s_a\trans{i}{j} =&\ s_{0} \reW{0}{0}{j-1}{-i}{j-1} \hspace{1cm}(\text{by~(\ref{eq:t13})})\\
=&\ s_0 I\downarrow^{j-1}_2 s_1 s_0 s_1 \Lambda_{2,2}^{-i} s_1 s_0 s_1 I\uparrow^{j-1}_2\\
=&\ I\downarrow^{j-1}_2 (s_0 s_1 s_0 s_1) \Lambda_{2,2}^{-i} s_1 s_0 s_1 I\uparrow^{j-1}_2 \hspace{1cm} (\text{$s_0 s_b = s_b s_0$ for $b\in [2,n]$})\\
=&\ I\downarrow^{j-1}_2 s_1 s_0 s_1 s_0 \Lambda_{2,2}^{-i} s_1 s_0 s_1 I\uparrow^{j-1}_2 \hspace{1cm} (\text{$s_0 s_1 s_0 s_1 = s_1 s_0 s_1 s_0$})\\
=&\ I\downarrow^{j-1}_2 s_1 s_0 s_1 \Lambda_{2,2}^{-i} (s_0 s_1 s_0 s_1 ) I\uparrow^{j-1}_2 \hspace{1cm} (\text{$s_0 s_b = s_b s_0$ for $b\in [2,n]$})\\
=&\ I\downarrow^{j-1}_2 s_1 s_0 s_1 \Lambda_{2,2}^{-i} ( s_1 s_0 s_1 s_0 ) I\uparrow^{j-1}_2 \hspace{1cm} (\text{$s_0 s_1 s_0 s_1 = s_1 s_0 s_1 s_0$})\\
=&\ I\downarrow^{j-1}_2 s_1 s_0 s_1 \Lambda_{2,2}^{-i} s_1 s_0 s_1 I\uparrow^{j-1}_2 s_0 \hspace{1cm} (\text{$s_0 s_1 s_0 s_1 = s_1 s_0 s_1 s_0$})\\
=&\ \reW{0}{0}{j-1}{-i}{j-1} s_0\\
=&\ \trans{s_a(i)}{s_a(j)}s_a \hspace{1cm}(\text{by~(\ref{eq:t13})}).
\end{align*}

\noindent{\bf Subcase~3.3.} $ 0 < a < a+1 < |i|$. Then $s_a(i) = i$, $s_a(j) = j$ and
\begin{align*}
s_a\trans{i}{j} =&\ s_a \reW{0}{0}{j-1}{-i}{j-1} \hspace{1cm}(\text{by~(\ref{eq:t13})})\\
=&\ s_a I\downarrow^{j-1}_{a+2} s_{a+1} s_a V_0^{a-1,a-1} s_a s_{a+1} \Lambda_{a+2,a+2}^{-i} s_{a+1}s_a V_0^{a-1,a-1} s_a s_{a+1} I\uparrow^{j-1}_{a+2}\\
=&\ I\downarrow^{j-1}_{a+2} (s_a s_{a+1} s_a) V_0^{a-1,a-1} s_a s_{a+1} \Lambda_{a+2,a+2}^{-i} s_{a+1}s_a V_0^{a-1,a-1} s_a s_{a+1} I\uparrow^{j-1}_{a+2}\\
&\ \hspace{5cm} (\text{$s_a s_b = s_b s_a$ for $b\in [a+2,n]$})\\
=&\ I\downarrow^{j-1}_{a+2} s_{a+1} s_a s_{a+1} V_0^{a-1,a-1} s_a s_{a+1} \Lambda_{a+2,a+2}^{-i} s_{a+1}s_a V_0^{a-1,a-1} s_a s_{a+1} I\uparrow^{j-1}_{a+2} \\
&\ \hspace{6.5cm} (\text{$s_a s_{a+1} s_a= s_{a+1} s_a s_{a+1}$})\\
=&\ I\downarrow^{j-1}_{a+2} s_{a+1} s_a  V_0^{a-1,a-1} (s_{a+1}s_a s_{a+1} ) \Lambda_{a+2,a+2}^{-i} s_{a+1}s_a V_0^{a-1,a-1} s_a s_{a+1} I\uparrow^{j-1}_{a+2} \\
&\ \hspace{5cm} (\text{$s_{a+1} s_b = s_b s_{a+1}$ for $b\in [0,a-1]$})\\
=&\ I\downarrow^{j-1}_{a+2} s_{a+1} s_a  V_0^{a-1,a-1} s_a s_{a+1}s_a \Lambda_{a+2,a+2}^{-i} s_{a+1}s_a V_0^{a-1,a-1} s_a s_{a+1} I\uparrow^{j-1}_{a+2} \\
&\ \hspace{6.5cm} (\text{$s_{a+1} s_a s_{a+1} = s_a s_{a+1} s_a$})\\
=&\ I\downarrow^{j-1}_{a+2} s_{a+1} s_a  V_0^{a-1,a-1} s_a s_{a+1} \Lambda_{a+2,a+2}^{-i} (s_a s_{a+1}s_a) V_0^{a-1,a-1} s_a s_{a+1} I\uparrow^{j-1}_{a+2} \\
&\ \hspace{5cm} (\text{$s_a s_b = s_b s_a$ for $b\in [a+2,n]$})\\
=&\ I\downarrow^{j-1}_{a+2} s_{a+1} s_a  V_0^{a-1,a-1} s_a s_{a+1} \Lambda_{a+2,a+2}^{-i} s_{a+1} s_a s_{a+1} V_0^{a-1,a-1} s_a s_{a+1} I\uparrow^{j-1}_{a+2} \\
&\ \hspace{6.5cm} (\text{$s_a s_{a+1} s_a = s_{a+1} s_a s_{a+1}$})\\
=&\ I\downarrow^{j-1}_{a+2} s_{a+1} s_a  V_0^{a-1,a-1} s_a s_{a+1} \Lambda_{a+2,a+2}^{-i} s_{a+1} s_a V_0^{a-1,a-1} (s_{a+1} s_a s_{a+1}) I\uparrow^{j-1}_{a+2} \\
&\ \hspace{5cm} (\text{$s_{a+1} s_b = s_b s_{a+1}$ for $b\in [0,a-1]$})\\
=&\ I\downarrow^{j-1}_{a+2} s_{a+1} s_a  V_0^{a-1,a-1} s_a s_{a+1} \Lambda_{a+2,a+2}^{-i} s_{a+1} s_a V_0^{a-1,a-1} s_a s_{a+1} s_a I\uparrow^{j-1}_{a+2} \\
&\ \hspace{6.5cm} (\text{$s_{a+1} s_a s_{a+1} = s_a s_{a+1} s_a$})\\
=&\ I\downarrow^{j-1}_{a+2} s_{a+1} s_a  V_0^{a-1,a-1} s_a s_{a+1} \Lambda_{a+2,a+2}^{-i} s_{a+1} s_a V_0^{a-1,a-1} s_a s_{a+1}  I\uparrow^{j-1}_{a+2} s_a \\
&\ \hspace{5cm} (\text{$s_a s_b = s_b s_a$ for $b\in [a+2,n]$})\\
=&\ \reW{0}{0}{j-1}{-i}{j-1} s_a \\
=&\ \trans{s_a(i)}{s_a(j)}s_a \hspace{1cm}(\text{by~(\ref{eq:t13})}).
\end{align*}

\noindent{\bf Subcase~3.4.} $ 0 < a < a+1 = |i|$. Then $s_a(i) = s_a(-(a+1)) = -a$, $s_a(j) = j$ and
\begin{align*}
s_a\trans{i}{j} =&\ s_a \reW{0}{0}{j-1}{a+1}{j-1} \hspace{1cm}(\text{by~(\ref{eq:t13})})\\
=&\ s_a I\downarrow^{j-1}_{a+2} s_{a+1}s_a V_0^{a-1,a-1} s_a s_{a+1} s_a V_0^{a-1,j-1}\\
=&\ I\downarrow^{j-1}_{a+2} (s_a s_{a+1}s_a) V_0^{a-1,a-1} s_a s_{a+1} s_a V_0^{a-1,j-1}
\hspace{1cm} (\text{$s_a s_b = s_b s_a$ for $b\in [a+2,n]$})\\
=&\ I\downarrow^{j-1}_{a+2} s_{a+1} s_a s_{a+1} V_0^{a-1,a-1} s_a s_{a+1} s_a V_0^{a-1,j-1} \hspace{1cm} (\text{$s_a s_{a+1} s_a = s_{a+1} s_a s_{a+1}$})\\
=&\ I\downarrow^{j-1}_{a+2} s_{a+1} s_a V_0^{a-1,a-1} (s_{a+1} s_a s_{a+1}) s_a V_0^{a-1,j-1}\hspace{1cm} (\text{$s_{a+1} s_b = s_b s_{a+1}$ for $b\in [0,a-1]$})\\
=&\ I\downarrow^{j-1}_{a+2} s_{a+1} s_a V_0^{a-1,a-1} s_a s_{a+1} s_a s_a V_0^{a-1,j-1}\hspace{1cm} (\text{$s_{a+1} s_a s_{a+1} = s_a s_{a+1} s_a$})\\
=&\ I\downarrow^{j-1}_{a+2} s_{a+1} s_a V_0^{a-1,a-1} s_a s_{a+1} V_0^{a-1,j-1}
\hspace{1cm} (\text{$s_a^2 = 1$})\\
=&\ I\downarrow^{j-1}_{a+2} s_{a+1} s_a V_0^{a-1,a-1} s_a s_{a+1} V_0^{a-1,a-1}
s_a s_{a+1} I\uparrow^{j-1}_{a+2}\\
=&\ I\downarrow^{j-1}_{a+2} s_{a+1} s_a V_0^{a-1,a-1} s_a V_0^{a-1,a-1}
(s_{a+1} s_a s_{a+1}) I\uparrow^{j-1}_{a+2}\\
& \hspace{6cm} (\text{$s_{a+1} s_b = s_b s_{a+1}$ for $b\in [0,a-1]$})\\
=&\ I\downarrow^{j-1}_{a+2} s_{a+1} s_a V_0^{a-1,a-1} s_a V_0^{a-1,a-1}
s_a s_{a+1}s_a I\uparrow^{j-1}_{a+2} \hspace{1cm} (\text{$s_{a+1} s_a s_{a+1} = s_a s_{a+1} s_a$})\\
=&\ I\downarrow^{j-1}_{a+2} s_{a+1} s_a V_0^{a-1,a-1} s_a V_0^{a-1,a-1}
s_a s_{a+1} I\uparrow^{j-1}_{a+2} s_a \hspace{1cm} (\text{$s_a s_b = s_b s_a$ for $b\in [a+2,n]$})\\
=&\ \reW{0}{0}{j-1}{a}{j-1} s_a\\
=&\ \trans{s_a(i)}{s_a(j)}s_a  \hspace{1cm}(\text{by~(\ref{eq:t13})}).
\end{align*}

\noindent{\bf Subcase~3.5.} $ 0 < a = |i|=j-1$. Then $s_a(i) = s_{j-1}(-(j-1)) = -j$ and $s_a(j) = s_{j-1}(j) = j-1$. Note that
\begin{align*}
s_a\trans{i}{j}s_a = s_{j-1} \trans{-(j-1)}{j} s_{j-1} = \trans{-j}{j-1},
\end{align*}
where the last equation holds by Definition~\ref{defn:trans}. Thus,
\[
s_a\trans{i}{j} = \trans{-j}{j-1}s_a = \trans{s_a(i)}{s_a(j)}s_a.
\]

\noindent{\bf Subcase~3.6.} $ 0 < a = |i|<j-1$. Then $s_a(i) = s_a(-a) = -(a+1)$, $s_a(j) = j$ and
\begin{align*}
s_a\trans{i}{j} =&\ s_a \reW{0}{0}{j-1}{a}{j-1} \hspace{1cm}(\text{by~(\ref{eq:t13})})\\
=&\ s_a I\downarrow^{j-1}_{a+2} s_{a+1} s_a V_0^{a-1,a-1} s_a V_0^{a-1,j-1}\\
=&\ I\downarrow^{j-1}_{a+2} (s_a s_{a+1} s_a) V_0^{a-1,a-1} s_a V_0^{a-1,j-1}
\hspace{1cm} (\text{$s_a s_b = s_b s_a$ for $b\in [a+2,n]$})\\
=&\ I\downarrow^{j-1}_{a+2} s_{a+1} s_a s_{a+1} V_0^{a-1,a-1} s_a V_0^{a-1,j-1}
\hspace{1cm} (\text{$s_a s_{a+1} s_a = s_{a+1} s_a s_{a+1}$})\\
=&\ I\downarrow^{j-1}_{a+2} s_{a+1} s_a V_0^{a-1,a-1} s_{a+1} s_a V_0^{a-1,j-1}
\hspace{1cm} (\text{$s_{a+1} s_b = s_b s_{a+1}$ for $b\in [0,a-1]$})\\
=&\ V_0^{j-1,a-1} s_{a+1} s_a V_0^{a-1,j-1}\\
=&\ V_0^{j-1,a-1} s_{a+1} s_a s_{a+1} s_{a+1} V_0^{a-1,j-1}
\hspace{1cm} (\text{$s_{a+1}^2=1$})\\
=&\ V_0^{j-1,a-1} s_a s_{a+1} s_a s_{a+1} V_0^{a-1,j-1}
\hspace{1cm} (\text{$s_{a+1} s_a s_{a+1} = s_a s_{a+1} s_a$})\\
=&\ V_0^{j-1,a-1} s_a s_{a+1} s_a s_{a+1} V_0^{a-1,a-1} s_a s_{a+1} I\uparrow^{j-1}_{a+2}\\
=&\ V_0^{j-1,a-1} s_a s_{a+1} s_a V_0^{a-1,a-1} s_{a+1} s_a s_{a+1} I\uparrow^{j-1}_{a+2} \hspace{1cm} (\text{$s_{a+1} s_b = s_b s_{a+1}$ for $b\in [0,a-1]$})\\
=&\ V_0^{j-1,a-1} s_a s_{a+1} s_a V_0^{a-1,a-1} s_a s_{a+1} s_a I\uparrow^{j-1}_{a+2}\hspace{1cm} (\text{$s_{a+1} s_a s_{a+1} = s_a s_{a+1} s_a$})\\
=&\ V_0^{j-1,a-1} s_a s_{a+1} s_a V_0^{a-1,a-1} s_a s_{a+1} I\uparrow^{j-1}_{a+2} s_a \hspace{1cm} (\text{$s_a s_b = s_b s_a$ for $b\in [a+2,n]$})\\
=&\ \reW{0}{0}{j-1}{a+1}{j-1} s_a\\
=&\ \trans{s_a(i)}{s_a(j)} s_a \hspace{1cm}(\text{by~(\ref{eq:t13})}).
\end{align*}

\noindent{\bf Subcase~3.7.} $ 0 < |i| <a <j-1$. Then $s_a(i) = i$, $s_a(j) = j$ and
\begin{align*}
s_a\trans{i}{j} =&\ s_a \reW{0}{0}{j-1}{-i}{j-1} \hspace{1cm}(\text{by~(\ref{eq:t13})})\\
=&\ s_a I\downarrow^{j-1}_{a+2} s_{a+1} s_a \reW{0}{0}{a-1}{-i}{a-1} s_a s_{a+1} I\uparrow^{j-1}_{a+2}\\
=&\ I\downarrow^{j-1}_{a+2} s_a s_{a+1} s_a \reW{0}{0}{a-1}{-i}{a-1} s_a s_{a+1} I\uparrow^{j-1}_{a+2} \hspace{1cm} (\text{$s_a s_b = s_b s_a$ for $b\in [a+2,n]$})\\
=&\ I\downarrow^{j-1}_{a+2} s_{a+1} s_a s_{a+1} \reW{0}{0}{a-1}{-i}{a-1} s_a s_{a+1} I\uparrow^{j-1}_{a+2}\hspace{1cm} (\text{$s_a s_{a+1} s_a = s_{a+1} s_a s_{a+1}$})\\
=&\ I\downarrow^{j-1}_{a+2} s_{a+1} s_a \reW{0}{0}{a-1}{-i}{a-1} s_{a+1} s_a s_{a+1} I\uparrow^{j-1}_{a+2}\hspace{1cm} (\text{$s_{a+1} s_b = s_b s_{a+1}$ for $b\in [0,a-1]$})\\
=&\ I\downarrow^{j-1}_{a+2} s_{a+1} s_a \reW{0}{0}{a-1}{-i}{a-1} s_a s_{a+1} s_a I\uparrow^{j-1}_{a+2} \hspace{1cm} (\text{$s_{a+1} s_a s_{a+1} = s_a s_{a+1} s_a$})\\
=&\ I\downarrow^{j-1}_{a+2} s_{a+1} s_a \reW{0}{0}{a-1}{-i}{a-1} s_a s_{a+1} I\uparrow^{j-1}_{a+2} s_a\hspace{1cm} (\text{$s_a s_b = s_b s_a$ for $b\in [a+2,n]$})\\
=&\ \reW{0}{0}{j-1}{-i}{j-1} s_a\\
=&\ \trans{s_a(i)}{s_a(j)} s_a \hspace{1cm}(\text{by~(\ref{eq:t13})}).
\end{align*}

\noindent{\bf Subcase~3.8.} $ 0 < |i| <a =j-1$. Then $s_a(i) = i$, $s_a(j) = s_{j-1}(j)=j-1$ and
\begin{align*}
s_a\trans{i}{j} \overset{(\ref{eq:t13})}{=} s_{j-1} \reW{0}{0}{j-1}{-i}{j-1} = s_{j-1} s_{j-1} \reW{0}{0}{j-2}{-i}{j-2} s_{j-1} \overset{(\ref{eq:Cgene})}{=}  \reW{0}{0}{j-2}{-i}{j-2} s_{j-1} \overset{(\ref{eq:t13})}{=} \trans{s_a(i)}{s_a(j)} s_a.
\end{align*}

\noindent{\bf Subcase~3.9.} $a=j<n$. Then $s_a(i) = i$, $s_a(j)=j+1$ and
\begin{align*}
s_a\trans{i}{j} \overset{(\ref{eq:t13})}{=} s_j \reW{0}{0}{j-1}{-i}{j-1} \overset{(\ref{eq:Cgene})}{=} s_j \reW{0}{0}{j-1}{-i}{j-1} s_j s_j = \reW{0}{0}{j}{-i}{j}s_j  \overset{(\ref{eq:t13})}{=} \trans{s_a(i)}{s_a(j)} s_a.
\end{align*}

\noindent{\bf Subcase~3.10.} $a=j=n$. Then $s_a(i) = i$ and $s_a(j) = s_n(n) = n+2$. It follows from Case~C in Remark~\mref{rk:3cases} that
\begin{equation}
\trans{s_a(i)}{s_a(j)} = \trans{i}{n+2} = s_n \trans{i}{n} s_n
\mlabel{eq:tin2}
\end{equation}
Hence,
\begin{align*}
s_a\trans{i}{j} \overset{(\ref{eq:t13})}{=} s_n \reW{0}{0}{n-1}{-i}{n-1}
\overset{(\ref{eq:Cgene})}{=} s_n \reW{0}{0}{n-1}{-i}{n-1} s_n s_n \overset{(\ref{eq:t13})}{=} s_n \trans{i}{n} s_n s_n \overset{(\ref{eq:tin2})}{=} \trans{s_a(i)}{s_a(j)} s_a.
\end{align*}

\noindent{\bf Subcase~3.11.} $a>j$. Then $s_a(i) = i$, $s_a(j) =j$ and
\begin{align*}
s_a\trans{i}{j}\overset{(\ref{eq:Cgene})}{=}\trans{i}{j}s_a = \trans{s_a(i)}{s_a(j)} s_a.
\end{align*}
This completes the proof.
\end{proof}

We are now in a position to combine the strong-cover description with the
type C transpositions introduced above.  The following theorem gives the
precise correspondence between a strong marked cover of symmetric cores
and the associated type C transposition, including an explicit recovery
of the marking.

\begin{theorem}
Let $C:=(\SMC{\tau}{c}{\kappa})$ be a strong marked cover with marking $m(C):=c$. Assume that $\tau = \mathfrak{a}(w)$, $\kappa=\mathfrak{a}(u)$ for some $w,u\in\tilc^0$. Then $u=w\trans{i}{j}$ for some $i,j\in\mathbb{Z}\setminus\{ q(n+1),q\in\mathbb{Z}\}$, such that
$$w(j) = u(i) = c+\lfloor c/n \rfloor+1.$$
\mlabel{thm:main}
\end{theorem}

\begin{proof}
By Definition~\ref{defn:SBO}, $w$ is a subword of $u$. It is equal to that $w^{-1}u$ is a reflection, which means
\[
w^{-1}u = v^{-1}s_i v \, \text{ for some }\, v\in\tilc, \, i\in[0,n].
\]
Besides, (\ref{eq:tfact3}) and Theorem~\ref{thm:reflection} imply that the reflections in $\tilc$ are precisely type C transpositions. Thus,
\begin{equation}
u=w\trans{i}{j} \overset{(\ref{eq:reflection})}{=} \trans{w(i)}{w(j)}w\,\text{ for some }\, i,j\in\mathbb{Z}\setminus\{ q(n+1),q\in\mathbb{Z}\}.
\mlabel{eq:uwtij}
\end{equation}
Recall that the marking $m(C)$ is defined as the content of the most southeast box in one of the ribbon in $\kappa/\tau$. By the definition of the edge sequence, the marked box in $\kappa/\tau $ is touched by the south step
$p_{c+\lfloor c/n \rfloor+1}$. On the other hand,
\[
\kappa = \mathfrak{a}(u) = u\cdot\emptyset \overset{(\ref{eq:uwtij})}{=} \trans{w(i)}{w(j)}w \cdot\emptyset = \trans{w(i)}{w(j)} \cdot \tau \overset{(\ref{eq:twoaction})}{=} \trans{w(i)}{w(j)} * \tau.
\]
Due to the definition of ribbon, we have
\[
p_{c+\lfloor c/n \rfloor+1}(\tau) = 0, \qquad  p_{c+\lfloor c/n \rfloor+1}(\kappa) = 1.
\]
It follows that $w(j) = c+\lfloor c/n \rfloor+1$. This completes the proof.
\end{proof}

The following example illustrates Theorem~\ref{thm:main} by explicitly recovering the type C transposition and the marking from a strong marked cover.

\begin{exam}
Let   $$n=2, \qquad \tau:=(4,3,2,1) = \mathfrak{a}(s_1s_0s_2s_1s_0)=:\mathfrak{a}(w)$$ and $$\kappa:= (5,3,3,1,1) = \mathfrak{a}(s_0s_1s_0s_2s_1s_0)=:\mathfrak{a}(u).$$
Then there is a strong marked cover (see Example~\ref{ex:SMC})
\begin{equation*}
\tau=
\begin{tikzpicture}[scale=.3,line width=0.5pt,baseline=(a.base)]
\draw[draw=black] (-1,-1) rectangle (0,0);
\draw[draw=black] (0,-1) rectangle (1,0);
\draw[draw=black] (-1,0) rectangle (0,1);
\draw[draw=black] (1,-1) rectangle (2,0);
\draw[draw=black] (-1,1) rectangle (0,2);
\draw[draw=black] (0,0) rectangle (1,1);
\draw[draw=black] (-1,2) rectangle (0,3);
\draw[draw=black] (2,-1) rectangle (3,0);
\draw[draw=black] (0,1) rectangle (1,2);
\draw[draw=black] (1,0) rectangle (2,1);
\node (a) [align=center] {\\[0pt] };
\end{tikzpicture}
\overset{-4}{\Rightarrow}
\begin{tikzpicture}[scale=.3,line width=0.5pt,baseline=(a.base)]
\draw[draw=black] (-1,-1) rectangle (0,0);
\draw[draw=black] (0,-1) rectangle (1,0);
\draw[draw=black] (-1,0) rectangle (0,1);
\draw[draw=black] (1,-1) rectangle (2,0);
\draw[draw=black] (-1,1) rectangle (0,2);
\draw[draw=black] (0,0) rectangle (1,1);
\draw[draw=black] (-1,2) rectangle (0,3);
\draw[draw=black] (2,-1) rectangle (3,0);
\draw[draw=black] (0,1) rectangle (1,2);
\draw[draw=black] (1,0) rectangle (2,1);
\filldraw[yellow!90,draw=black] (3,-1) rectangle (4,0);
\filldraw[yellow!90,draw=black] (-1,3) rectangle (0,4);\node at(-0.5,3.5){\tiny\( \star \)};
\filldraw[yellow!90,draw=black] (1,1) rectangle (2,2);
\node (a) [align=center] {\\[0pt] };
\end{tikzpicture}
=\kappa,
\end{equation*}
with marking $c=-4$. Note that
\begin{equation*}
c+\lfloor c/n \rfloor+1 = -5.
\end{equation*}
Taking $$i,j\in\mathbb{Z}\setminus\{ q(n+1),q\in\mathbb{Z}\}\,\text{ such that }\, w(j) = u(i) = -5,$$
we have $i = -13, j=1$ and so $w(i) =-7, w(j) =-5$. This confirms the validity of Theorem~\ref{thm:main} via
\begin{align*}
u = s_0w \overset{(\ref{eq:tfact3})}{=} \trans{-1}{1} w \overset{(\ref{eq:tfact2})}{=} \trans{-7}{-5} w.
\end{align*}
\end{exam}

The correspondence in Theorem~\ref{thm:main} extends naturally from a single strong marked cover to chains of such covers, leading to permutation versions of strong marked strips and strong marked tableaux.

\begin{remark}
By Theorem~\ref{thm:main}, we obtain the $\tilc^0$ version of the strong marked strip and the strong marked tableau.
\end{remark}

\section{An action of type C permutations on partitions}\mlabel{s:action}
In Subsection~\ref{ss:PerCorePar}, we introduced an action of $\tilc$ on
$\Cco$ and a bijection between $\Cco$ and $\Cpar$. In this section, we
transport this action to $\Cpar$ and describe it explicitly in combinatorial
terms. We begin by recalling some notation from~\cite{LM}.

\begin{enumerate}[label=(\roman*)]
\item For $d\in\mathbb{Z}$, set
\[
D_d:= \{(i,j)\in\mathbb{Z}_{\geq1}\times \mathbb{Z}_{\geq1} \mid j-i=d  \}.
\]

\item A {\em $2n$-string} is a sequence of boxes $c_1,\ldots,c_s$, whose boxes lie in $D_{r+2n},\ldots,D_{r+2sn}$. Note that all boxes in one $2n$-string have the same residue.

\item For any two boxes $a$ and $b$ with $b$ south-east of $a$, let $a\land b$ denote the box directly south of $a$ and directly west of $b$.
\end{enumerate}

The following lemma recalls a basic structural property of addable and
removable corners of a core, which will be used to analyze the action of
the simple reflections.

\begin{lemma}{\rm (\cite[Property~17]{LM})}
Let $\kappa\in\co{2n}$.
\begin{enumerate}
\item If $\kappa$ has a removable corner of residue $i$, then the collection of all removable corners of $\kappa$ with residue $i$ forms a $2n$-string.
\mlabel{it:istring1}

\item If $\kappa$ has an addable corner of residue $i$, then the collection of all addable corners of $\kappa$ with residue $i$ forms a $2n$-string.
\mlabel{it:istring2}
\end{enumerate}
\mlabel{lem:istring}
\end{lemma}

Combining Lemma~\ref{lem:istring} with the symmetry of the cores in $\Cco$,
we obtain the following description of the boxes added or removed by the
action of a simple reflection.

\begin{coro}
Let $\kappa\in\Cco$ and $i\in[0,n]$.
If $$\dg{s_i\cdot \kappa} \supsetneq \dg{\kappa} \ ({\rm resp.} \dg{s_i\cdot \kappa} \subsetneq \dg{\kappa}),$$ then $s_i\cdot\kappa/\kappa$ (resp. $\kappa/s_i\cdot\kappa$) is composed of two $2n$-strings with residue $i$ and $2n-i \mod 2n$, respectively.  Moreover, there is a bijection between these two $2n$-strings via $\tau:(a,b)\mapsto (b,a)$.
\mlabel{coro:2string}
\end{coro}

\begin{proof}
Assume $\dg{s_i\cdot \kappa} \supsetneq \dg{\kappa}$. Then $\kappa$ has an addable corner $(a,b)$ of residue $i$ for some $a,b$. By the symmetry of $\Cco$, $(b,a)$ is an addable corner of residue $2n-i$ {\rm mod} $2n$. Hence, it follows from Lemma~\ref{lem:action} and Lemma~\ref{lem:istring}~(\ref{it:istring2}) that $s_i\cdot\kappa/\kappa$ consists of two $2n$-strings with residue $i$ and $2n-i$ {\rm mod} $2n$ respectively, and the bijection is natural.
Interchanging $\kappa\leftrightarrow s_i\cdot \kappa$ and using $s_i^2=1$ proves the case of $\dg{s_i\cdot \kappa} \subsetneq \dg{\kappa}$.
\end{proof}

Observe that for $i=0$ or $i=n$, the two $2n$-strings appearing in Corollary~\mref{coro:2string} are in fact identical.
The preceding description of the added and removed $2n$-strings allows us
to determine precisely which hook lengths cross the $2n$-bound under the
action of a simple reflection.

\begin{prop}
Let $\kappa\in\Cco$ and $i\in[0,n]$.
\begin{enumerate}
\item If $c_1,\ldots,c_s$ and $c_1',\ldots,c_s'$ are two $2n$-string of addable corners of $\kappa$ with residues $i$ and $2n-i$ {\rm mod} $2n$ respectively, then the boxes $$c_1\land c_2,\ldots,c_{s-1}\land c_s, c_1'\land c_2',\ldots,c_{s-1}'\land c_s'$$ are the only boxes whose hook length is $2n$-bounded in $\dg{\kappa}$ but exceeds $2n$ in $\dg{s_i\cdot\kappa}$.
\mlabel{it:stringhook1}

\item If $c_1,\ldots,c_s$ and $c_1',\ldots,c_s'$ are two $2n$-string of removable corners of $\kappa$ with residues $i$ and $2n-i$ {\rm mod} $2n$ respectively, then the boxes $$c_1\land c_2,\ldots,c_{s-1}\land c_s, c_1'\land c_2',\ldots,c_{s-1}'\land c_s'$$ are the only boxes whose hook length exceeds $2n$ in $\dg{\kappa}$ but is $2n$-bounded in $\dg{s_i\cdot\kappa}$.
\mlabel{it:stringhook2}
\end{enumerate}
\mlabel{prop:stringhook}
\end{prop}

\begin{proof}
We only need to prove Item~(\ref{it:stringhook1}), since Item~(\ref{it:stringhook2}) can be obtained by replacing $\kappa\leftrightarrow s_i\cdot\kappa$ in Item~(\ref{it:stringhook1}).
Note that, by the definition of $\Cco$, neither $\dg{\kappa}$ nor
$\dg{s_i\cdot\kappa}$ contains a box of hook length $2n$.
Thus, if there exists a box whose hook length is $2n$-bounded in $\kappa$ but exceeds $2n$ in $s_i\cdot\kappa$, then the hook length of this box increased by at least $2$ during $\kappa\to s_i\cdot\kappa$. If $\kappa$ has an addable corner of residue $i$, then $\dg{s_i\cdot \kappa} \supsetneq \dg{\kappa}$ due to Lemma~\ref{lem:action}. By Corollary~\ref{coro:2string}, $s_i\cdot\kappa$ is obtained from $\kappa$ by adding $2n$-strings $c_1,\ldots,c_s$ and $c_1',\ldots,c_s'$. Hence, the hook length of any box $c$ in $\dg{s_i\cdot \kappa}$ is at least 2 longer than that in $\dg{\kappa}$ only if $c= c_a\land c_b$ or $c_a'\land c_b'$ for some $a,b\in[1,s]$.

Consider the hook length of such $c$ in $\dg{\kappa}$, and assume $c_m=(i_m,j_m)$ and $c_m'=(i_m',j_m')$ for $m\in[1,s]$. Then
\begin{align*}
\hl{c_m\land c_{m+1}} =&\ (i_m-i_{m+1}-1) + (j_{m+1}-j_m-1) +1\\
=&\ (j_{m+1}-i_{m+1})-(j_m-i_m)-1\\
=&\ 2n-1,
\end{align*}
and for $1\leq a<b-1<s$,
\begin{align*}
\hl{c_a\land c_{b}} =&\ (i_a-i_b-1) + (j_b-j_a-1) +1\\
=&\ (j_b-i_b)-(j_a-i_a)-1\\
>&\ 2n.
\end{align*}
By the symmetry of $\kappa$,
$$\hl{c_m'\land c_{m+1}'}=2n-1, \qquad \hl{c_a'\land c_{b}'}>2n\,\text{ for }\, 1\leq a<b-1<s.$$
It follows that the boxes
$c_m\land c_{m+1}$ and $c_m'\land c_{m+1}'$,
for $m\in[1,s-1]$, are precisely those with $2n$-bounded hook lengths
in $\dg{\kappa}$.
This completes the proof.
\end{proof}

We can now combine the preceding analysis of $2n$-strings and hook lengths
with the bijection $\mathfrak{p}:\Cco\to\Cpar$ to obtain an explicit
description of the induced group action of $\tilc$ on bounded partitions.


\begin{theorem}
Let $\lambda\in\Cpar$ and $\kappa:=\mathfrak{p}^{-1}(\lambda)$. Then
\begin{equation*}
s_i\cdot\kappa =
\left\{
\begin{array}{ll}
\mathfrak{p}^{-1}(\lambda+\epsilon_a), & \quad \text{{\rm if} $\dg{\kappa}\subsetneq \dg{s_i\cdot\kappa}$},\\
\mathfrak{p}^{-1}(\lambda-\epsilon_r), & \quad \text{{\rm if} $\dg{\kappa}\supsetneq \dg{s_i\cdot\kappa}$},\\
\kappa, & \quad \text{{\rm otherwise}},
\end{array}
\right.
\end{equation*}
where
\[
a:=\max\{ c \mid (c,d)\,\text{ is an addable corner of $\kappa$ with residue $i$ or $2n-i$ {\rm mod} $2n$, and $d-c\geq0$} \},
\]
\[
r:=\max\{ c \mid (c,d)\,\text{ is a removable corner of $\kappa$ with residue $i$ or $2n-i$ {\rm mod} $2n$, and $d-c\geq0$} \}.
\]
\mlabel{thm:silambda}
\end{theorem}

\begin{proof}
If neither $\dg{\kappa}\subsetneq \dg{s_i\cdot\kappa}$ nor $\dg{\kappa}\supsetneq \dg{s_i\cdot\kappa}$ holds, then the ``otherwise" case is obvious by Lemma~\ref{lem:action}~(\ref{it:action3}). Suppose $\dg{\kappa}\subsetneq \dg{s_i\cdot\kappa}$. By Lemma~\ref{coro:2string}, $s_i\cdot\kappa/\kappa$ is composed of two $2n$-strings with residue $i$ and $2n-i$ {\rm mod} $2n$ respectively. Denote such $2n$-strings by $c_1,\ldots,c_s$ and $c_1',\ldots,c_s'$. Proposition~\ref{prop:stringhook}~(\ref{it:stringhook1}) reveals that, the boxes $$c_1\land c_2,\ldots,c_{s-1}\land c_s, c_1'\land c_2',\ldots,c_{s-1}'\land c_s'$$ are the only boxes whose hook length is $2n$-bounded in $\dg{\kappa}$ but exceeds $2n$ in $\dg{s_i\cdot\kappa}$. Claim: if $c_m$ (resp. $c_m'$) is a box in row $a$ for some $m\in[1,s]$, then $c_{m-1}\land c_m$ (resp. $c_{m-1}'\land c_m'$) lies in or above the main diagonal. Follows from this claim, we obtain that $s_i$ acts on $\kappa$ by increasing the box with $2n$-bounded hook length only in row $a$, since ${\rm skew}^{\perp}(\kappa)$ exchanged by removing $c_i\land c_{i+1}$ and $c_j'\land c_{j+1}'$ and adding $c_x$ and $c_y'$ all lie below the main diagonal. Thus, $s_i\cdot\kappa=\mathfrak{p}^{-1}(\lambda+\epsilon_a)$ by the description of $\mathfrak{p}$ in Lemma~\ref{lem:bij2}.

Finally, we prove the claim. By symmetry of these two $2n$-strings, we may assume that $c_m$ is exactly the box in row $a$ for some $m\in[1,s]$. Suppose $c_{m-1}\land c_m$ lies below the main diagonal. Set $c_m := (a,j)$, $c_{m-1} = (i,j')$ with $j'-i\leq0$. Then $c_{m-1}\land c_m = (a,j')$ with $j'-a>0$. The symmetry of two $2n$-strings implies that $$
(j,a) = c_{k-1}', \qquad (j',i)=c_{k}'\,\text{ for some }\,k\in[2,s].
$$
It follows that $j'$ has the following properties: $(j',i)=c_k'$ is an
addable corner of $\kappa$ of residue $2n-i$ modulo $2n$,
$i-j'\geq 0$, and $j'>a$. This contradicts the maximality of $a$.
This proves the claim and completes the proof in the case
$\dg{\kappa}\subsetneq\dg{s_i\cdot\kappa}$.

Replacing $\kappa\leftrightarrow s_i\cdot \kappa$ and using $s_i^2=1$ prove the case of $\dg{\kappa}\supsetneq\dg{s_i\cdot \kappa}$.
\end{proof}

The following example illustrates Theorem~\ref{thm:silambda} by computing
explicitly the action of a simple reflection on a symmetric core and on
the corresponding bounded partition.

\begin{exam}
Let $n=2$, $i=1$ and $\kappa:= (5,4,3,2,1)\in\Cco$. Then
\begin{equation*}
s_i\cdot\kappa=
\begin{tikzpicture}[scale=.3,line width=0.5pt,baseline=(a.base)]
\draw[draw=black] (-1,-1) rectangle (0,0);
\draw[draw=black] (0,-1) rectangle (1,0);
\draw[draw=black] (-1,0) rectangle (0,1);
\draw[draw=black] (1,-1) rectangle (2,0);
\draw[draw=black] (-1,1) rectangle (0,2);
\draw[draw=black] (0,0) rectangle (1,1);
\draw[draw=black] (-1,2) rectangle (0,3);
\draw[draw=black] (2,-1) rectangle (3,0);
\draw[draw=black] (0,1) rectangle (1,2);
\draw[draw=black] (1,0) rectangle (2,1);
\draw[draw=black] (3,-1) rectangle (4,0);
\draw[draw=black] (-1,3) rectangle (0,4);
\draw[draw=black] (1,1) rectangle (2,2);
\draw[draw=black] (0,2) rectangle (1,3);
\draw[draw=black] (2,0) rectangle (3,1);
\filldraw[blue!60,draw=black] (-1,4) rectangle (0,5);
\filldraw[blue!30,draw=black] (4,-1) rectangle (5,0);
\filldraw[blue!60,draw=black] (3,0) rectangle (4,1);
\filldraw[blue!30,draw=black] (0,3) rectangle (1,4);
\filldraw[blue!30,draw=black] (2,1) rectangle (3,2);
\filldraw[blue!60,draw=black] (1,2) rectangle (2,3);
\node (a) [align=center] {\\[-1000pt] };
\end{tikzpicture}
,
\end{equation*}
where the light blue boxes are the $2n$-strings with residue $i$, and the dark blue boxes are the $2n$-strings with residue $2n-i$ {\rm mod} $2n$
(see Example~\ref{ex:bij1}). Note that $a=3$.
Let $\lambda:=\mathfrak{p}(\kappa)$. By Lemma~\ref{lem:bij2}, $\lambda=(3,3,1)$ via
\[
\begin{tikzpicture}[scale=.3,line width=0.5pt,baseline=(a.base)]
\draw[draw=black] (-1,-4) rectangle (0,-3);\node at(-0.5,-3.5){\tiny\( 9 \)};
\draw[draw=black] (0,-4) rectangle (1,-3);\node at(0.5,-3.5){\tiny\( 7 \)};
\draw[draw=black] (-1,-3) rectangle (0,-2);\node at(-0.5,-2.5){\tiny\( 7 \)};
\draw[draw=black] (1,-4) rectangle (2,-3);\node at(1.5,-3.5){\tiny\( 5 \)};
\draw[draw=black] (-1,-2) rectangle (0,-1);\node at(-0.5,-1.5){\tiny\( 5 \)};
\draw[draw=black] (0,-3) rectangle (1,-2);\node at(0.5,-2.5){\tiny\( 5 \)};
\draw[draw=black] (-1,-1) rectangle (0,0);\node at(-0.5,-0.5){\tiny\( 3 \)};
\draw[draw=black] (2,-4) rectangle (3,-3);\node at(2.5,-3.5){\tiny\( 3 \)};
\draw[draw=black] (0,-2) rectangle (1,-1);\node at(0.5,-1.5){\tiny\( 3 \)};
\draw[draw=black] (1,-3) rectangle (2,-2);\node at(1.5,-2.5){\tiny\( 3 \)};
\draw[draw=black] (3,-4) rectangle (4,-3);\node at(3.5,-3.5){\tiny\( 1 \)};
\draw[draw=black] (-1,0) rectangle (0,1);\node at(-0.5,0.5){\tiny\( 1 \)};
\draw[draw=black] (1,-2) rectangle (2,-1);\node at(1.5,-1.5){\tiny\( 1 \)};
\draw[draw=black] (0,-1) rectangle (1,0);\node at(0.5,-0.5){\tiny\( 1 \)};
\draw[draw=black] (2,-3) rectangle (3,-2);\node at(2.5,-2.5){\tiny\( 1 \)};
\node (a) [align=center] {\\[20pt] };
\end{tikzpicture}
\to
\begin{tikzpicture}[scale=.3,line width=0.5pt,baseline=(a.base)]
\draw[draw=black] (-1,-4) rectangle (0,-3);\node at(-0.5,-3.5){\tiny\( 9 \)};
\draw[draw=black] (0,-4) rectangle (1,-3);\node at(0.5,-3.5){\tiny\( 7 \)};
\draw[draw=black] (-1,-3) rectangle (0,-2);\node at(-0.5,-2.5){\tiny\( 7 \)};
\draw[draw=black] (1,-4) rectangle (2,-3);\node at(1.5,-3.5){\tiny\( 5 \)};
\draw[draw=black] (-1,-2) rectangle (0,-1);\node at(-0.5,-1.5){\tiny\( 5 \)};
\draw[draw=black] (0,-3) rectangle (1,-2);\node at(0.5,-2.5){\tiny\( 5 \)};
\filldraw[green!90,draw=black] (-1,-1) rectangle (0,0);\node at(-0.5,-0.5){\tiny\( 3 \)};
\filldraw[green!30,draw=black] (2,-4) rectangle (3,-3);\node at(2.5,-3.5){\tiny\( 3 \)};
\filldraw[green!90,draw=black] (0,-2) rectangle (1,-1);\node at(0.5,-1.5){\tiny\( 3 \)};
\filldraw[green!30,draw=black] (1,-3) rectangle (2,-2);\node at(1.5,-2.5){\tiny\( 3 \)};
\filldraw[green!30,draw=black] (3,-4) rectangle (4,-3);\node at(3.5,-3.5){\tiny\( 1 \)};
\filldraw[green!90,draw=black] (-1,0) rectangle (0,1);\node at(-0.5,0.5){\tiny\( 1 \)};
\filldraw[green!90,draw=black] (1,-2) rectangle (2,-1);\node at(1.5,-1.5){\tiny\( 1 \)};
\filldraw[green!90,draw=black] (0,-1) rectangle (1,0);\node at(0.5,-0.5){\tiny\( 1 \)};
\filldraw[green!30,draw=black] (2,-3) rectangle (3,-2);\node at(2.5,-2.5){\tiny\( 1 \)};
\node (a) [align=center] {\\[20pt] };
\end{tikzpicture}
\to
\begin{tikzpicture}[scale=.3,line width=0.5pt,baseline=(a.base)]
\draw[draw=black] (-1,-4) rectangle (0,-3);
\filldraw[green!30,draw=black] (0,-4) rectangle (1,-3);
\filldraw[green!30,draw=black] (1,-4) rectangle (2,-3);
\draw[draw=black] (-1,-3) rectangle (0,-2);
\filldraw[green!30,draw=black] (0,-3) rectangle (1,-2);
\filldraw[green!30,draw=black] (1,-3) rectangle (2,-2);
\draw[draw=black] (-1,-2) rectangle (0,-1);
\node (a) [align=center] {\\[20pt] };
\end{tikzpicture}
\, ,
\]
and $\mathfrak{p}(s_i\cdot\kappa) = (3,3,2) = \lambda+\epsilon_a$ via
\[
\begin{tikzpicture}[scale=.3,line width=0.5pt,baseline=(a.base)]
\draw[draw=black] (-1,-4) rectangle (0,-3);\node at(-0.5,-3.5){\tiny\( 11 \)};
\draw[draw=black] (0,-4) rectangle (1,-3);\node at(0.5,-3.5){\tiny\( 9 \)};
\draw[draw=black] (-1,-3) rectangle (0,-2);\node at(-0.5,-2.5){\tiny\( 9 \)};
\draw[draw=black] (1,-4) rectangle (2,-3);\node at(1.5,-3.5){\tiny\( 7 \)};
\draw[draw=black] (-1,-2) rectangle (0,-1);\node at(-0.5,-1.5){\tiny\( 7 \)};
\draw[draw=black] (0,-3) rectangle (1,-2);\node at(0.5,-2.5){\tiny\( 7 \)};
\draw[draw=black] (-1,-1) rectangle (0,0);\node at(-0.5,-0.5){\tiny\( 5 \)};
\draw[draw=black] (2,-4) rectangle (3,-3);\node at(2.5,-3.5){\tiny\( 5 \)};
\draw[draw=black] (0,-2) rectangle (1,-1);\node at(0.5,-1.5){\tiny\( 5 \)};
\draw[draw=black] (1,-3) rectangle (2,-2);\node at(1.5,-2.5){\tiny\( 5 \)};
\draw[draw=black] (3,-4) rectangle (4,-3);\node at(3.5,-3.5){\tiny\( 3 \)};
\draw[draw=black] (-1,0) rectangle (0,1);\node at(-0.5,0.5){\tiny\( 3 \)};
\draw[draw=black] (1,-2) rectangle (2,-1);\node at(1.5,-1.5){\tiny\( 3 \)};
\draw[draw=black] (0,-1) rectangle (1,0);\node at(0.5,-0.5){\tiny\( 3 \)};
\draw[draw=black] (2,-3) rectangle (3,-2);\node at(2.5,-2.5){\tiny\( 3 \)};
\draw[draw=black] (4,-4) rectangle (5,-3);\node at(4.5,-3.5){\tiny\( 1 \)};
\draw[draw=black] (3,-3) rectangle (4,-2);\node at(3.5,-2.5){\tiny\( 1 \)};
\draw[draw=black] (2,-2) rectangle (3,-1);\node at(2.5,-1.5){\tiny\( 1 \)};
\draw[draw=black] (1,-1) rectangle (2,0);\node at(1.5,-0.5){\tiny\( 1 \)};
\draw[draw=black] (0,-0) rectangle (1,1);\node at(0.5,0.5){\tiny\( 1 \)};
\draw[draw=black] (-1,1) rectangle (0,2);\node at(-0.5,1.5){\tiny\( 1 \)};
\node (a) [align=center] {\\[20pt] };
\end{tikzpicture}
\to
\begin{tikzpicture}[scale=.3,line width=0.5pt,baseline=(a.base)]
\draw[draw=black] (-1,-4) rectangle (0,-3);\node at(-0.5,-3.5){\tiny\( 11 \)};
\draw[draw=black] (0,-4) rectangle (1,-3);\node at(0.5,-3.5){\tiny\( 9 \)};
\draw[draw=black] (-1,-3) rectangle (0,-2);\node at(-0.5,-2.5){\tiny\( 9 \)};
\draw[draw=black] (1,-4) rectangle (2,-3);\node at(1.5,-3.5){\tiny\( 7 \)};
\draw[draw=black] (-1,-2) rectangle (0,-1);\node at(-0.5,-1.5){\tiny\( 7 \)};
\draw[draw=black] (0,-3) rectangle (1,-2);\node at(0.5,-2.5){\tiny\( 7 \)};
\draw[draw=black] (-1,-1) rectangle (0,0);\node at(-0.5,-0.5){\tiny\( 5 \)};
\draw[draw=black] (2,-4) rectangle (3,-3);\node at(2.5,-3.5){\tiny\( 5 \)};
\draw[draw=black] (0,-2) rectangle (1,-1);\node at(0.5,-1.5){\tiny\( 5 \)};
\draw[draw=black] (1,-3) rectangle (2,-2);\node at(1.5,-2.5){\tiny\( 5 \)};
\filldraw[green!30,draw=black] (3,-4) rectangle (4,-3);\node at(3.5,-3.5){\tiny\( 3 \)};
\filldraw[green!90,draw=black] (-1,0) rectangle (0,1);\node at(-0.5,0.5){\tiny\( 3 \)};
\filldraw[green!90,draw=black] (1,-2) rectangle (2,-1);\node at(1.5,-1.5){\tiny\( 3 \)};
\filldraw[green!90,draw=black] (0,-1) rectangle (1,0);\node at(0.5,-0.5){\tiny\( 3 \)};
\filldraw[green!30,draw=black] (2,-3) rectangle (3,-2);\node at(2.5,-2.5){\tiny\( 3 \)};
\filldraw[green!30,draw=black] (4,-4) rectangle (5,-3);\node at(4.5,-3.5){\tiny\( 1 \)};
\filldraw[green!30,draw=black] (3,-3) rectangle (4,-2);\node at(3.5,-2.5){\tiny\( 1 \)};
\filldraw[green!30,draw=black] (2,-2) rectangle (3,-1);\node at(2.5,-1.5){\tiny\( 1 \)};
\filldraw[green!90,draw=black] (1,-1) rectangle (2,0);\node at(1.5,-0.5){\tiny\( 1 \)};
\filldraw[green!90,draw=black] (0,-0) rectangle (1,1);\node at(0.5,0.5){\tiny\( 1 \)};
\filldraw[green!90,draw=black] (-1,1) rectangle (0,2);\node at(-0.5,1.5){\tiny\( 1 \)};
\node (a) [align=center] {\\[20pt] };
\end{tikzpicture}
\to
\begin{tikzpicture}[scale=.3,line width=0.5pt,baseline=(a.base)]
\draw[draw=black] (-1,-4) rectangle (0,-3);
\filldraw[green!30,draw=black] (0,-4) rectangle (1,-3);
\filldraw[green!30,draw=black] (1,-4) rectangle (2,-3);
\draw[draw=black] (-1,-3) rectangle (0,-2);
\filldraw[green!30,draw=black] (0,-3) rectangle (1,-2);
\filldraw[green!30,draw=black] (1,-3) rectangle (2,-2);
\draw[draw=black] (-1,-2) rectangle (0,-1);
\filldraw[green!30,draw=black] (0,-2) rectangle (1,-1);
\node (a) [align=center] {\\[20pt] };
\end{tikzpicture}
\, .
\]
\end{exam}

\smallskip

\noindent
{\bf Acknowledgements}: This work is supported by the Natural Science Foundation of Gansu Province (25JRRA644), Innovative Fundamental Research Group Project of Gansu Province (23JRRA684) and Longyuan Young Talents of Gansu Province.

\noindent
{\bf Declaration of interests.} The authors have no conflicts of interest to disclose.

\noindent
{\bf Data availability.} Data sharing is not applicable as no new data were created or analyzed.

\end{document}